\documentclass[12pt,a4paper]{amsart}
\usepackage{amsmath}
\usepackage{amsthm}
\usepackage{amsfonts}
\usepackage{amssymb}
\usepackage{amscd}
\usepackage{mathrsfs}
\usepackage{tikz-cd}
\usepackage[english]{babel}
\usepackage{verbatim}
\usepackage[shortlabels]{enumitem}
\usepackage[latin1]{inputenc}

\newtheoremstyle{mio}%
	{}{} 
	{\itshape}{} 
	{\bfseries}{.}{ } 
	{#1 #2\thmnote{\mdseries~ #3}} 
\theoremstyle{mio}
\newtheorem{teor}{Theorem}[section]
\newtheorem{cor}[teor]{Corollary}
\newtheorem{prop}[teor]{Proposition}
\newtheorem{lemma}[teor]{Lemma}
\newtheorem{defin}[teor]{Definition}

\theoremstyle{definition}
\newtheorem{ex}[teor]{Example}
\newtheorem{oss}[teor]{Remark}

\renewcommand{\star}{\ast}

\newcommand{\insfracid}{\mathcal{F}}

\newcommand{\inssubmod}{\mathbf{F}}
\newcommand{\insstarid}{\mathcal{F}}

\newcommand{\insstar}{\mathrm{Star}}
\newcommand{\inssemiprime}{\mathrm{Sp}}

\newcommand{\inssmstar}{\mathrm{(S)Star}}
\newcommand{\inssemistar}{\mathrm{SStar}}

\newcommand{\insstarstab}{\insstar_{st}}

\newcommand{\Spec}{\mathrm{Spec}}
\newcommand{\Max}{\mathrm{Max}}

\newcommand{\Jac}{\mathrm{Jac}}

\newcommand{\Inv}{\mathrm{Inv}}
\newcommand{\Prin}{\mathrm{Prin}}
\newcommand{\Cl}{\mathrm{Cl}}
\newcommand{\Pic}{\mathrm{Pic}}

\newcommand{\eql}{\end{equation*}\begin{equation*}}

\newcommand{\ins}[1]{\mathbb{#1}}

\newcommand{\insZ}{\ins{Z}}

\newcommand{\insQ}{\ins{Q}}
\newcommand{\insR}{\ins{R}}

\newcommand{\valut}{\mathbf{v}}

\newcommand{\ortog}[1]{#1^\perp}

\title{Jaffard families and localizations of star operations}
\author{Dario Spirito}
\email{spirito@mat.uniroma3.it}
\address{Dipartimento di Matematica e Fisica, Universit\`a degli Studi
``Roma Tre'', Roma, Italy}
\date{\today}

\keywords{Star operations, overrings, Jaffard families, Pr\"ufer domains, class group}
\subjclass[2010]{13A15, 13C20, 13F05, 13F30, 13G05}

\begin{document}

\begin{abstract}
We generalize the concept of localization of a star operation to flat overrings; subsequently, we investigate the possibility of representing the set $\insstar(R)$ of star operations on $R$ as the product of $\insstar(T)$, as $T$ ranges in a family of overrings of $R$ with special properties. We then apply this method to study the set of star operations on a Pr\"ufer domain $R$, in particular the set of stable star operations and the star-class groups of $R$.
\end{abstract}

\maketitle

\section{Introduction}
Recently, the study of star operations, initiated by the works of Krull \cite{krull_idealtheorie} and Gilmer \cite[Chapter 32]{gilmer}, has focused on studying the whole set $\insstar(R)$ of star operations on $R$, and in particular its cardinality. Using as a starting point the characterization of domains with $|\insstar(R)|=1$ due to Heinzer \cite{heinzer_d=v}, Houston, Mimouni and Park have devoted a series of papers \cite{twostar,houston_noeth-starfinite,hmp_finite,starnoeth_resinfinito} to this study, obtaining, among other results, a characterization of Pr\"ufer domains with two star operations \cite[Theorem 3.3]{twostar} and the precise determination of $|\insstar(R)|$ on some classes of one-dimensional Noetherian domains \cite{houston_noeth-starfinite,starnoeth_resinfinito}. Their work is based -- at least partly -- on the concept of \emph{localization} of finite-type star operations to localizations of the ring.

The purpose of this paper is to generalize the concept of localization of a star operation $\star$, by avoiding (when possible) the hypothesis that $\star$ is of finite type and by considering, instead of localizations, flat overrings of the base ring $R$. In particular, we will prove that, if $R$ admits a family of overrings with certain properties (precisely, a \emph{Jaffard family} \cite[Section 6.3]{{fontana_factoring}}) then $\insstar(R)$ can be represented as a cartesian product of $\insstar(T)$, as $T$ ranges in this family, and that this representation preserves the main properties of the star operations.

We then specialize to the case of Pr\"ufer domain, when this approach is complemented by the possibility, in certain cases, to link star operations on $R$ with star operations on a quotient of $R$. This method allows one to obtain a better grasp of several properties, like being a stable operation (Proposition \ref{prop:prufer-stab}), and to describe the star-class group of $R$ in terms of the class groups of some localizations of $R$.

\section{Preliminaries and notation}
Let $R$ be an integral domain with quotient field $K$, and denote by $\insfracid(R)$ the set of fractional ideals of $R$. A \emph{star operation} on $R$ is a map $\star:\insfracid(R)\longrightarrow\insfracid(R)$, $I\mapsto I^\star$ such that, for every $I,J\in\insfracid(R)$ and $x\in K$,
\begin{enumerate}[(a)]
\item $I\subseteq I^\star$;
\item $(I^\star)^\star=I^\star$;
\item if $I\subseteq J$, then $I^\star\subseteq J^\star$;
\item $R^\star=R$;
\item $(xI)^\star=x\cdot I^\star$.
\end{enumerate}
The set of star operations on $R$ is denoted by $\insstar(R)$. An ideal $I$ is a \emph{$\star$-ideal} if $I=I^\star$.

Similarly, a \emph{semistar operation} on $R$ is a map $\star:\inssubmod(R)\longrightarrow\inssubmod(R)$ (where $\inssubmod(R)$ is the set of $R$-submodules of $K$) satisfying the previous properties, except for $R^\star=R$; if $\star$ verifies also the latter, then it is said to be a \emph{(semi)star operation}. We indicate the sets of semistar and (semi)star operations by $\inssemistar(R)$ and $\inssmstar(R)$, respectively. A \emph{semiprime operation} is a map $c$, from the set of integral ideals of $R$ to itself, that satisfies the first four properties of star operations and, moreover, such that $xI^\star\subseteq(xI)^\star$ for every $x\in R$.

A star operation is said to be:
\begin{itemize}
\item \emph{of finite type} if, for every $I$,
\begin{equation*}
I^\star=\bigcup\{J^\star\mid J\subseteq I,~J\text{~is finitely generated}\};
\end{equation*}

\item \emph{semifinite} if any proper $\star$-ideal $I$ is contained in a prime $\star$-ideal;

\item \emph{stable} if $(I\cap J)^\star=I^\star\cap J^\star$ for all ideals $I,J$;

\item \emph{spectral} if it is in the form $I^\star=\bigcap\{IR_P\mid P\in\Delta\}$ for some $\Delta\subseteq\Spec(R)$; equivalently, $\star$ is spectral if and only if it is stable and semifinite \cite[Theorem 4]{anderson_overrings_1988};

\item \emph{endlich arithmetisch brauchbar} (\emph{eab} for short) if, for every nonzero finitely generated ideals $F,G,H$ such that $(FG)^\star\subseteq(FH)^\star$, we have $G^\star\subseteq H^\star$; if this property holds for arbitrary nonzero fractional ideals $G,H$ (but $F$ still finitely generated) then $\star$ is said to be \emph{arithmetisch brauchbar} (\emph{ab} for short);

\item \emph{Noetherian} if any set $\{I_\alpha\mid \alpha\in A\}$ of proper $\star$-ideals has a maximum, or equivalently if and only if every ascending chain of $\star$-closed ideals stabilizes. (More commonly, under this hypothesis $R$ is said to be \emph{$\star$-Noetherian} \cite{zaf_ACCstar}.)
\end{itemize}

The set of star operations has a natural order, such that $\star_1\leq\star_2$ if and only if $I^{\star_1}\subseteq I^{\star_2}$ for every ideal $I$, or equivalently if and only if every $\star_2$-closed ideal is also $\star_1$-closed. Under this order, $\insstar(R)$ becomes a complete lattice, where the minimum is the identity (usually denoted by $d$) and the maximum the $v$-operation (or \emph{divisorial closure}) $I\mapsto(R:(R:I))$.

If $R$ is an integral domain, an \emph{overring} of $R$ is a ring $T$ contained between $R$ and its quotient field $K$. A family $\Theta$ of overrings of $R$ is \emph{locally finite} (or \emph{of finite character}) if every $x\in K\setminus\{0\}$ (or, equivalently, every $x\in R\setminus\{0\}$) is a nonunit in only finitely many $T\in\Theta$. The ring $R$ itself is said to be of finite character if $\{R_M:M\in\Max(R)\}$ is a family of finite character.

A flat overring of $R$ is an overring that is flat as a $R$-module. If $T$ is a flat overring, then $(I_1\cap\cdots\cap I_n)T=I_1T\cap\cdots\cap I_nT$ for every $I_1,\ldots,I_n\in\inssubmod(R)$, and $(I:J)T=(IT:JT)$ for every $I,J\in\inssubmod(R)$ with $J$ finitely generated \cite[Theorem 7.4]{matsumura} (see also \cite[Proposition 2]{compact-intersections}).

\section{Extendable star operations}\label{sect:extension}
The starting point is the notion of localization of a star operation, originally defined in \cite{twostar}. We shall adopt a more general and more abstract approach.
\begin{defin}\label{defin:starloc}
Let $R$ be an integral domain and $T$ a flat overring of $R$. We say that a star operation $\star\in\insstar(R)$ is \emph{extendable to $T$} if the map
\begin{equation}
\begin{aligned}
\star_T\colon \insfracid(T) & \longrightarrow \insfracid(T)\\
IT & \longmapsto I^\star T
\end{aligned}
\end{equation}
is well-defined (where $I$ is a fractional ideal of $R$).
\end{defin}

\begin{oss}\label{oss:def}
~\begin{enumerate}
\item If $T$ is flat over $R$, then every fractional ideal of $T$ is an extension of a fractional ideal of $R$ (since, if $J$ is an integral ideal of $T$, $J=(J\cap R)T$); therefore, $\star_T$ is (potentially) defined on all of $\insfracid(T)$.
\item\label{oss:def:primi} If $T$ is flat over $R$ and $P$ is a prime of $R$ such that $PT\neq T$, then $PT$ is a prime ideal of $T$. Indeed, let $Q$ be a minimal prime of $PT$. By the previous point, $Q=(Q\cap R)T$; suppose $P\subsetneq Q\cap R$. By \cite[Theorem 2]{richamn_generalized-qr}, $T_Q=R_{Q\cap R}$, and thus $PT_Q$ is not minimal over $QT_Q=(Q\cap R)T_Q$, a contradiction. Note that the equality $T_Q=R_{Q\cap R}$ also shows that there is at most one $Q\in\Spec(T)$ over any $P\in\Spec(R)$.
\item When $T=S^{-1}R$ is a localization of $R$ and $\star$ is of finite type, Definition \ref{defin:starloc} coincides with the definition of $\star_S$ given in \cite[Proposition 2.4]{twostar}.
\item When $T=R_P$ for some $P\in\Spec(R)$, we will sometimes denote $\star_T$ with $\star_P$.
\end{enumerate}
\end{oss}

The following proposition shows the basic properties of extendability.
\begin{prop}\label{prop:starloc:basic}
Let $R$ be an integral domain, let $\star\in\insstar(R)$ and let $T$ be a flat overring of $R$.
\begin{enumerate}[(a)]
\item\label{prop:starloc:basic:star} If $\star$ is extendable to $T$, then $\star_T$ is a star operation.
\item\label{prop:starloc:basic:equiv} $\star$ is extendable to $T$ if and only if $I^\star T=J^\star T$ whenever $IT=JT$.
\item\label{prop:starloc:basic:id} The identity star operation $d$ is always extendable, and $d_T$ is the identity on $T$.
\item\label{prop:starloc:basic:ft} If $\star$ is of finite type, then it is extendable to $T$, and $\star_T$ is of finite type.
\end{enumerate}
\end{prop}

Note that, if $T$ is a localization of $R$, point \ref{prop:starloc:basic:ft} is proved in \cite[Proposition 2.4]{twostar}.

\begin{proof}
\ref{prop:starloc:basic:star} and \ref{prop:starloc:basic:id} are obvious, while \ref{prop:starloc:basic:equiv} is just a reformulation of Definition \ref{defin:starloc}.

For \ref{prop:starloc:basic:ft}, by symmetry it is enough to show that $J^\star T\subseteq I^\star T$, or equivalently that $1\in(I^\star T:J^\star T)$. Since $\star$ is of finite type,
\begin{equation*}
(I^\star T:J^\star T)=\left(I^\star T:\left(\sum_{\substack{L\subseteq J\\ L\text{~finitely generated}}}L^\star\right)T\right)= \left(I^\star T:\sum_{\substack{L\subseteq J\\ L\text{~fin. gen.}}}L^\star T\right)=\eql
=\bigcap_{\substack{L\subseteq J\\ L\text{~fin. gen.}}}(I^\star T:L^\star T)\supseteq\bigcap_{\substack{L\subseteq J\\ L\text{~fin. gen.}}}(I^\star:L^\star)T.
\end{equation*}
By properties of star operations, $(I^\star:L^\star)=(I^\star:L)$; since $L$ is finitely generated and $T$ is flat, it follows that, for every $L$,
\begin{equation*}
(I^\star:L^\star)T=(I^\star:L)T=(I^\star T:LT)
\end{equation*}
which contains 1 since $LT\subseteq JT=IT\subseteq I^\star T$. Hence, $1\in(I^\star T:J^\star T)$, as requested.
\end{proof}

\begin{ex}\label{ex:ad}
Not every star operation is extendable: let $R$ be an almost Dedekind domain which is not Dedekind (i.e., a one-dimensional non-Noetherian domain such that $R_M$ is a discrete valuation ring for every $M\in\Max(R)$), and let $P$ be a non-finitely generated prime ideal of $R$. Then $P$ is not divisorial \cite[Lemma 4.1.8]{fontana_libro}, and thus the $v$-operation is not extendable to $R_P$, since otherwise $(PR_P)^{v_P}=P^vR_P=R_P$, while the unique star operation on $R_P$ is the identity.
\end{ex}

Beside being of finite type, extension preserves the main properties of a star operation.
\begin{prop}\label{prop:estensione-prop}
Let $R$ be a domain and $T$ be a flat overring of $R$; suppose $\star\in\insstar(R)$ is extendable to $T$. If $\star$ is stable (respectively, spectral, Noetherian) then so is $\star_T$.
\end{prop}
\begin{proof}
Suppose $\star$ is stable, and let $I_1:=J_1T$, $I_2:=J_2T$ be ideals of $T$, where $J_1$ and $J_2$ are ideals of $R$. Then,
\begin{equation*}
\begin{array}{rcl}
(I_1\cap I_2)^{\star_T} & = & (J_1T\cap J_2T)^{\star_T}=[(J_1\cap J_2)T]^{\star_T}=\\
& = & (J_1\cap J_2)^\star T=(J_1^\star\cap J_2^\star)T=J_1^\star T\cap J_2^\star T=I_1^{\star_T}\cap I_2^{\star_T}
\end{array}
\end{equation*}
and thus $\star_T$ is stable.

If $\star$ is spectral, it is stable, and thus so is $\star_T$. Let now $I$ be a proper $\star_T$-closed ideal of $T$, and let $J:=I\cap R$; then, $JT=(I\cap R)T=I$, and thus $J^\star\subseteq I^{\star_T}\cap R=I\cap R=J$, so that $J$ is a $\star$-ideal. By definition, there is a $\Delta\subseteq\Spec(R)$ such that $\star=\star_\Delta$; hence,
\begin{equation*}
J=J^\star=\bigcap_{P\in\Delta}JR_P=\bigcap_{P\in\Delta}(I\cap R)R_P=\bigcap_{P\in\Delta}IR_P\cap R_P.
\end{equation*}
In particular, there is a $P\in\Delta$ such that $1\notin IR_P=ITR_P$; hence, there is a $Q\in\Spec(TR_P)$ such that $ITR_P\subseteq Q$. We claim that $Q_0:=Q\cap T$ is a prime $\star_T$-ideal containing $I$. Indeed, $I\subseteq ITR_P\cap T\subseteq Q\cap T=Q_0$; moreover, since $Q\cap R=Q_0\cap R\subseteq P$, $Q_0=LT$ for some prime ideal $L$ of $T$ contained in $P$ (Remark \ref{oss:def}\ref{oss:def:primi}), and thus
\begin{equation*}
Q_0^{\star_T}=L^\star T\subseteq (LR_P\cap R)T=LT=Q_0.
\end{equation*}
Therefore, $\star_T$ is also semifinite, and by \cite[Theorem 4]{anderson_overrings_1988} it is spectral.

Suppose $\star$ is Noetherian, and let $\{I_\alpha T\mid\alpha\in A\}$ be an ascending chain of $\star_T$-ideals. Then, $\{I_\alpha^\star\mid\alpha\in A\}$ is an ascending chain of $\star$-ideals, which has to stabilize at $I_{\overline{\alpha}}$. Hence, the original chain stabilizes at $I_{\overline{\alpha}}T$, and $\star_T$ is Noetherian.
\end{proof}

Extendability works well with the order structure of $\insstar(R)$.
\begin{prop}\label{prop:starloc:ordine}
Let $R$ be an integral domain and $T$ be a flat overring of $R$. Let $\star_1,\star_2,\{\star_\lambda\mid\lambda\in\Lambda\}$ be star operations that are extendable to $T$.
\begin{enumerate}[(a)]
\item\label{prop:starloc:ordine:leq} If $\star_1\leq\star_2\in\insstar(R)$, then $(\star_1)_T\leq(\star_2)_T$.
\item\label{prop:starloc:ordine:wedge} $\star_1\wedge\star_2$ is extendable to $T$ and $(\star_1\wedge\star_2)_T=(\star_1)_T\wedge(\star_2)_T$.
\item\label{prop:starloc:ordine:supft} If each $\star_\lambda$ is of finite type, then $\sup_\lambda\star_\lambda$ is extendable to $T$ and $(\sup_\lambda\star_\lambda)_T=\sup_\lambda(\star_\lambda)_T$.
\end{enumerate}
\end{prop}
\begin{proof}
\ref{prop:starloc:ordine:leq} If $\star_1\leq\star_2$, then $I^{\star_1}\subseteq I^{\star_2}$ for every fractional ideal $I$, and thus $(I^{\star_1}T)\subseteq(I^{\star_2}T)$. Using the definition of $\star_T$, we get $(\star_1)_T\leq(\star_2)_T$.

\ref{prop:starloc:ordine:wedge} Let $I$ be an ideal of $R$. By definition, $I^{\star_1\wedge\star_2}=I^{\star_1}\cap I^{\star_2}$, so that
\begin{equation*}
(IT)^{(\star_1\wedge\star_2)_T}=(I^{\star_1\wedge\star_2})T=(I^{\star_1}\cap I^{\star_2})T=\eql
=I^{\star_1}T\cap I^{\star_2}T=(IT)^{(\star_1)_T}\cap (IT)^{(\star_2)_T}=(IT)^{(\star_1)_T\wedge(\star_2)_T}
\end{equation*}
and thus $(\star_1\wedge\star_2)_T=(\star_1)_T\wedge(\star_2)_T$.

\ref{prop:starloc:ordine:supft} Let $\star:=\sup_\lambda\star_\lambda$. Since each $\star_\lambda$ is of finite type, so is $\star$ \cite[p.1628]{anderson_examples}, and thus $\star$ is extendable to $T$ by Proposition \ref{prop:starloc:basic}\ref{prop:starloc:basic:ft}. Moreover, again by \cite[p.1628]{anderson_examples}, $I^\star=\sum I^{\star_1\circ\cdots\circ\star_n}$, as $(\star_1,\ldots,\star_n)$ ranges among the finite strings of elements of $\{\star_\lambda\mid\lambda\in\Lambda\}$ (here $\star_1\circ\cdots\circ\star_n$ indicates simply the composition of $\star_1,\ldots,\star_n$); therefore,
\begin{equation*}
I^\star T=\left(\sum I^{\star_1\circ\cdots\circ\star_n}\right)T=\sum I^{\star_1\circ\cdots\circ\star_n}T.
\end{equation*}
We claim that $I^{\star_1\circ\cdots\circ\star_n}T=(IT)^{(\star_1)_T\circ\cdots\circ(\star_n)_T}$; we proceed by induction. The case $n=1$ is just the definition of the extension; suppose the claim holds for $m<n$. Then,
\begin{equation*}
I^{\star_1\circ\cdots\circ\star_n}T=(I^{\star_1})^{\star_2\circ\cdots\circ\star_n}T= (I^{\star_1}T)^{(\star_2)_T\circ\cdots\circ(\star_n)_T}=(IT)^{(\star_1)_T\circ\cdots\circ(\star_n)_T}
\end{equation*}
as claimed. Thus,
\begin{equation*}
I^\star T=\sum(IT)^{(\star_1)_T\circ\cdots\circ(\star_n)_T}=(IT)^{\sup_\lambda(\star_\lambda)_T}
\end{equation*}
the last equality coming from \cite[p.1628]{anderson_examples} and Proposition \ref{prop:starloc:basic}\ref{prop:starloc:basic:ft}. Hence, $\star=\sup_\lambda(\star_\lambda)_T$.
\end{proof}

Extendability is also transitive:
\begin{prop}\label{prop:ext-transitivo}
Let $R$ be a domain and $T_1\subseteq T_2$ be two flat overrings of $R$. If $\star\in\insstar(R)$ is extendable to $T_1$ and $\star_{T_1}$ is extendable to $T_2$, then $\star$ is extendable to $T_2$, and $\star_{T_2}=(\star_{T_1})_{T_2}$.
\end{prop}
\begin{proof}
Note first that if $T_2$ is flat over $R$ then it is flat over $T_1$, and thus it makes sense to speak of the extendability of $\star_{T_1}$. For every ideal $I$ of $R$, we have
\begin{equation*}
I^\star T_2=(I^\star T_1)T_2=(IT_1)^{\star_{T_1}}T_2=(IT_1T_2)^{(\star_{T_1})_{T_2}}=(IT_2)^{(\star_{T_1})_{T_2}}
\end{equation*}
and thus if $IT_2=JT_2$ then $I^\star T_2=J^\star T_2$, so that $\star$ is extendable to $T_2$. The previous calculation also shows that $\star_{T_2}=(\star_{T_1})_{T_2}$.
\end{proof}

\begin{prop}
Let $R$ be an integral domain and $T$ be a flat overring of $R$. Let $\Delta:=\{M\cap R\mid M\in\Max(T)\}$. If $\star\in\insstar(R)$ is extendable to $R_P$, for every $P\in\Delta$, then it is extendable to $T$.
\end{prop}
\begin{proof}
Let $I,J$ be ideals of $R$ such that $IT=JT$. Let $P\in\Delta$ and let $M$ be the (necessarily unique -- see Remark \ref{oss:def}(\ref{oss:def:primi})) maximal ideal of $T$ such that $M\cap R=P$. Then, $T_M=R_P$, and since $\star$ is extendable to $R_P$ we have $I^\star R_P=J^\star R_P$. It follows that 
\begin{equation*}
I^\star T=\bigcap_{P\in\Delta}I^\star R_P=\bigcap_{P\in\Delta}J^\star R_P=J^\star T,
\end{equation*}
and thus $\star$ is extendable to $T$.
\end{proof}

\begin{cor}\label{cor:def-extstar}
Let $R$ be a domain, and let $\star\in\insstar(R)$. The following are equivalent:
\begin{enumerate}[(i)]
\item $\star$ is extendable to $R_P$, for every $P\in\Spec(R)$;
\item $\star$ is extendable to every flat overring of $R$.
\end{enumerate}
\end{cor}

Note that condition (i) of the above corollary cannot be replaced by the version that considers only maximal ideals of $T$: indeed, if $(R,M)$ is local, then clearly every star operation is extendable to $R_M$, but it would be implausible that every star operation is extendable to every localization. We can build an explicit counterexample tweaking slightly \cite[Remark 2.5(3)]{twostar}. Let $R:=\insZ_{p\insZ}+X\insQ(\sqrt{2})[[X]]$ (where $p$ is a prime number). Then, $R$ is a two-dimensional local domain, with maximal ideal $M:=p\insZ_{p\insZ}+X\insQ(\sqrt{2})[[X]]$; let $P:=X\insQ(\sqrt{2})[[X]]$. We claim that the $v$-operation is not extendable to $R_P=\insQ+P$. Let $A:=X(\insQ+P)$ and $B:=XR$: then, $AR_P=BR_P=A$, but $A^vR_P=P$ while $B^vR_P=BR_P\neq P$.

\section{Jaffard families}\label{sect:Jaffard}
The concept of Jaffard family was introduced and studied in \cite[Section 6.3]{fontana_factoring}.
\begin{defin}
Let $R$ be a domain and $\Theta$ be a set of overrings of $R$ such that the quotient field of $R$ is not in $\Theta$. We say that $\Theta$ is a \emph{Jaffard family on $R$} if, for every integral ideal $I$ of $R$,
\begin{itemize}
\item $R=\bigcap_{T\in\Theta}T$;
\item $\Theta$ is locally finite;
\item $I=\prod_{T\in\Theta}(IT\cap R)$;
\item if $T\neq S$ are in $\Theta$, then $(IT\cap R)+(IS\cap R)=R$.
\end{itemize}

We say that an overring $T$ of $R$ is a \emph{Jaffard overring} of $R$ if $T$ belongs to a Jaffard family of $R$.
\end{defin}

Note that, by the second axiom, if $I\neq(0)$ then $IT=T$ for all but finitely many $T\in\Theta$, so that the product $I=\prod_{T\in\Theta}(IT\cap R)$ is finite.

The next propositions collect the properties of Jaffard families that we will be using.
\begin{prop}[{\protect\cite[Theorem 6.3.1]{fontana_factoring}}]\label{prop:jaffard:basic}
Let $R$ be an integral domain with quotient field $K$, and let $\Theta$ be a Jaffard family on $R$. For each $T\in\Theta$, let $\ortog{\Theta}(T):=\bigcap\{U\in\Theta\mid U\neq T\}$.
\begin{enumerate}[(a)]
\item\label{prop:jaffard:basic:complete} $\Theta$ is complete (i.e., $I=\bigcap\{IT\mid T\in\Theta\}$ for every ideal $I$ of $R$).
\item\label{prop:jaffard:basic:partitionSpec} For each $P\in\Spec(R)$, $P\neq(0)$, there is a unique $T\in\Theta$ such that $PT\neq T$.
\item\label{prop:jaffard:basic:flat} For each $T\in\Theta$, both $T$ and $\ortog{\Theta}(T)$ are flat over $R$.
\item\label{prop:jaffard:basic:complindip} For each $T\in\Theta$, we have $T\cdot\ortog{\Theta}(T)=K$.
\end{enumerate}
\end{prop}

\begin{prop}\label{prop:caratt-jaffard}
Let $\Theta$ be a family of flat overrings of the domain $R$, and let $K$ be the quotient field of $R$. Then, $\Theta$ is a Jaffard family if and only if it is complete, locally finite and $TS=K$ for all $T,S\in\Theta$, $T\neq S$.
\end{prop}
\begin{proof}
If $\Theta$ is a Jaffard family, the properties follow by the definition and Proposition \ref{prop:jaffard:basic}. Conversely, suppose $\Theta$ verifies the three properties, let $I\neq(0)$ be an ideal of $R$ and let $T\neq S$ be members of $\Theta$. If $IT\cap R$ and $IS\cap R$ are not coprime, then there would be a prime $P$ of $R$ containing both; since $\Theta$ is complete, it would follow that both $IT\cap R$ and $IS\cap R$ survive in some $A\in\Theta$. In particular, without loss of generality, $A\neq T$; however,
\begin{equation*}
(IT\cap R)A=ITA\cap A=IK\cap A=A,
\end{equation*}
a contradiction. Therefore, $(IT\cap R)+(IS\cap R)=R$. Moreover, $I=\bigcap\{IT\cap R\mid T\in\Theta\}=(IT_1\cap R)\cap\cdots\cap(IT_n\cap R)$ by local finiteness; since the $IT_i\cap R$ are coprime, their intersection is equal to their product, and thus $I=(IT_1\cap R)\cdots(IT_n\cap R)$.
\end{proof}

\begin{oss}\label{oss:Matlis}
Any Jaffard family $\Theta$ defines a partition on $\Max(R)$, where each class is composed by the $M\in\Max(R)$ such that $MT\neq T$ for some fixed $T\in\Theta$. In particular, $T=\bigcap R_M$, as $M$ ranges in the class relative to $M$; hence, different Jaffard families define different partitions. In particular, a local domain has only one Jaffard family, namely $\{R\}$, and a semilocal domain has only a finite number of Jaffard families.

However, not every partition of $\Max(R)$ can arise in this way. For example, let $\Theta$ be a Jaffard family and let $M,N\in\Max(R)$; by Proposition \ref{prop:jaffard:basic}\ref{prop:jaffard:basic:partitionSpec}, there are unique overrings $T,U\in\Theta$ such that $MT\neq T$ and $NU\neq U$. If there is a nonzero prime $P\subseteq M\cap N$, then $PT\neq T$ and $PU\neq U$; therefore, again by Proposition \ref{prop:jaffard:basic}\ref{prop:jaffard:basic:partitionSpec}, it must be $T=U$. 
\end{oss}

A \emph{$h$-local domain} is an integral domain $R$ such that $\Max(R)$ is locally finite and such that every prime ideal $P$ is contained in only one maximal ideal. In this case, $\{R_M\mid M\in\Max(R)\}$ is a Jaffard family of $R$; conversely, if $\{R_M\mid M\in\Max(R)\}$ is a Jaffard family, then $\Max(R)$ is locally finite (by definition) and each prime is contained in only one maximal ideal (by Proposition \ref{prop:jaffard:basic}\ref{prop:jaffard:basic:partitionSpec}), and thus $R$ is $h$-local. Many properties of the Jaffard families can be seen as generalizations of the corresponding properties of $h$-local domains; the following proposition is an example (compare \cite[Proposition 3.1]{olberding_globalizing}).
\begin{prop}\label{prop:integintersect}
Let $R$ be a domain and $T$ be a Jaffard overring of $R$. Then:
\begin{enumerate}[(a)]
\item\label{prop:integintersect:a} for every family $\{X_\alpha:\alpha\in A\}$ of $R$-submodules of $K$ with nonzero intersection, we have $\left(\bigcap_{\alpha\in A}X_\alpha\right)T=\bigcap_{\alpha\in A}X_\alpha T$;
\item\label{prop:integintersect:b} if $\{I_\alpha:\alpha\in A\}$ is a family of integral ideals of $R$ with nonzero intersection such that $\left(\bigcap_{\alpha\in A}I_\alpha\right)T\neq T$, then $I_{\overline{\alpha}}T\neq T$ for some $\overline{\alpha}\in A$.
\end{enumerate}
\end{prop}
\begin{proof}
\ref{prop:integintersect:a} Let $\Theta$ be a Jaffard family of $R$ such that $T\in\Theta$. Then, by the flatness of $T$,
\begin{equation*}
\left(\bigcap_{\alpha\in A}X_\alpha\right)T= 
\left(\bigcap_{\alpha\in A}\bigcap_{U\in\Theta}X_\alpha U\right)T=
\left(\bigcap_{U\in\Theta}\bigcap_{\alpha\in A}X_\alpha U\right)T=\eql
=\left(\bigcap_{U\in\Theta\setminus\{T\}}\bigcap_{\alpha\in A}X_\alpha U\right)T\cap \bigcap_{\alpha\in A}X_\alpha T=K\cap\bigcap_{\alpha\in A}X_\alpha T
\end{equation*}
since $\bigcap_{U\in\Theta\setminus\{T\}}\bigcap_{\alpha\in A}X_\alpha U$ is a $\ortog{\Theta}(T)$-module, and thus its product with $T$ is equal to $K$ by Proposition \ref{prop:jaffard:basic}\ref{prop:jaffard:basic:complindip}.

(a $\Longrightarrow$ b). Suppose $\left(\bigcap_{\alpha\in A}I_\alpha\right)T\neq T$. Since $\left(\bigcap_{\alpha\in A}I_\alpha\right)T\subseteq T$, then 1 is not contained in the left hand side. By \ref{prop:integintersect:a}, 1 is not contained in $\bigcap_{\alpha\in A}I_\alpha T$, i.e., there is a $\overline{\alpha}$ such that $1\notin I_{\overline{\alpha}} T$, and thus $I_{\overline{\alpha}} T\neq T$.
\end{proof}

\section{Jaffard families and star operations}\label{sect:jaff-star}
The reason why we introduced Jaffard families is that they provide a way to decompose $\insstar(R)$ as a product of spaces of star operations of overrings of $T$. Before reaching this objective (Theorem \ref{teor:star-jaffard}) we show that weaker properties can lead to a decomposition of at least a subset of $\insstar(R)$.

\begin{prop}\label{prop:indip-rhoinj}
Let $R$ be an integral domain with quotient field $K$. Let $\Theta$ be a set of flat overrings of $R$ such that $R=\bigcap\{T\mid T\in\Theta\}$ and such that $AB=K$ whenever $A,B\in\Theta$ and $A\neq B$. Then, there is an injective order-preserving map 
\begin{equation*}
\begin{aligned}
\rho_\Theta\colon\prod_{T\in\Theta}\insstar(T) & \longrightarrow\insstar(R) \\
(\star^{(T)})_{T\in\Theta} & \longmapsto\bigwedge_{T\in\Theta}\star^{(T)},
\end{aligned}
\end{equation*}
where $\bigwedge_{T\in\Theta}\star^{(T)}$ is the map such that
\begin{equation*}
I\mapsto\bigcap_{T\in\Theta}(IT)^{\star^{(T)}}
\end{equation*}
for every fractional ideal $I$ of $R$.
\end{prop}
\begin{proof}
Let $(\star_T)_{T\in\Theta}\in\prod_{T\in\Theta}\insstar(T)$, and let $\star:=\rho_\Theta((\star^{(T)})_{T\in\Theta})$. Since $\bigcap_{T\in\Theta}T=R$, the map $\star$ is a star operation; moreover, it is clear that if $\star_1^{(T)}\leq\star_2^{(T)}$ for all $T$ then $\rho_\Theta(\star_1^{(T)})\leq\rho_\Theta(\star_2^{(T)})$. Hence, $\rho_\Theta$ is well-defined and order-preserving; we need to show that it is injective.

Suppose it is not; then, $\star:=\rho_\Theta(\star_1^{(T)})=\rho_\Theta(\star_2^{(T)})$ for some families of star operations such that $\star_1^{(U)}\neq\star_2^{(U)}$ for some $U\in\Theta$. There is an integral ideal $J$ of $U$ such that $J^{\star_1^{(U)}}\neq J^{\star_2^{(U)}}$; let $I:=J\cap R$. Since $U$ is flat, for both $i=1$ and $i=2$ we have
\begin{equation*}
I^\star U=\left[\bigcap_{T\in\Theta}(IT)^{\star_i^{(T)}}\right]U=(IU)^{\star_i^{(U)}}U\cap \left[\bigcap_{T\in\Theta\setminus\{U\}}(IT)^{\star_i^{(T)}}\right]U.
\end{equation*}

If $T\neq U$, then, since $T$ is flat,
\begin{equation*}
(IT)^{\star_i^{(T)}}=((J\cap R)T)^{\star_i^{(T)}}=(JT\cap T)^{\star_i^{(T)}}.
\end{equation*}
However, $JT=JUT=K$ since $UT=K$ (by hypothesis); therefore, $(IT)^{\star_i^{(T)}}=T$, and (since $I\subseteq U$)
\begin{equation*}
I^\star U=(IU)^{\star_i^{(U)}}U\cap\left[\bigcap_{T\in\Theta\setminus\{U\}}T\right]U= (IU)^{\star_i^{(U)}}U\cap\left[\bigcap_{T\in\Theta}T\right]U=\eql =(IU)^{\star_i^{(U)}}\cap RU=(IU)^{\star_i^{(U)}}=J^{\star_i^{(U)}}
\end{equation*}
for both $i=1$ and $i=2$. However, this contradicts the choice of $J$; hence, $\rho_\Theta$ is injective.
\end{proof}

If $\Theta$ is a Jaffard family, the previous proposition can be strengthened. We need two lemmas.
\begin{lemma}\label{lemma:intersez-ritorno}
Let $R$ be a domain with quotient field $K$, and let $\Theta$ be a Jaffard family on $R$. For every $U\in\Theta$, let $J_U$ be a $U$-submodule of $K$, and define $J:=\bigcap_{U\in\Theta}J_U$. If $J\neq(0)$, then for every $T\in\Theta$ we have $JT=J_T$.
\end{lemma}
\begin{proof}
By Proposition \ref{prop:integintersect}\ref{prop:integintersect:a}, we have
\begin{equation*}
JT=\left(\bigcap_{U\in\Theta}J_U\right)T= \bigcap_{U\in\Theta}J_UT.
\end{equation*}
If $U\neq T$, then $J_UT=J_UUT=J_UK=K$; therefore, $JT=J_TT=J_T$.
\end{proof}

The next lemma can be seen as a generalization of \cite[Theorem 6.2.2(2)]{fontana_factoring} and \cite[Lemma 2.3]{warfield}.
\begin{lemma}\label{lemma:Tcolon}
Let $R$ be an integral domain, $T$ be a Jaffard overring of $R$, and let $I,J\in\inssubmod(R)$ such that $(I:J)\neq(0)$. Then, $(I:J)T=(IT:JT)$.
\end{lemma}
\begin{proof}
It is enough to note that $(I:J)=\bigcap_{j\in J}j^{-1}I\neq(0)$, and apply Proposition \ref{prop:integintersect}\ref{prop:integintersect:a}.
\end{proof}

\begin{teor}\label{teor:star-jaffard}
Let $R$ be an integral domain and let $\Theta$ be a Jaffard family on $R$. Then, every $\star\in\insstar(R)$ is extendable to every $T\in\Theta$, and the maps
\begin{equation*}
\begin{aligned}
\lambda_\Theta\colon \insstar(R) & \longrightarrow \prod_{T\in\Theta}\insstar(T)\\
\star & \longmapsto (\star_T)_{T\in\Theta}
\end{aligned}
\quad\text{and}\quad
\begin{aligned}
\rho_\Theta\colon \prod_{T\in\Theta}\insstar(T) & \longrightarrow \insstar(R)\\
(\star^{(T)})_{T\in\Theta} & \longmapsto \bigwedge_{T\in\Theta}\star^{(T)}
\end{aligned}
\end{equation*}
(where $\bigwedge_{T\in\Theta}\star^{(T)}$ is defined as in Proposition \ref{prop:indip-rhoinj}) are order-preserving bijections between $\insstar(R)$ and $\prod\{\insstar(T)\mid T\in\Theta\}$.
\end{teor}
\begin{proof}
We first show that every $\star\in\insstar(R)$ is extendable. Let $T\in\Theta$ and let $I,J$ be ideals of $R$ such that $IT=JT$. Then, using Lemma \ref{lemma:Tcolon}, we have
\begin{equation*}
(I^\star T:J^\star T)=(I^\star:J^\star)T=(I^\star:J)T=(I^\star T:JT)
\end{equation*}
and, since $JT=IT\subseteq I^\star T$, we have $1\in(I^\star T:J^\star T)$, so that $J^\star T\subseteq I^\star T$. Symmetrically, $I^\star T\subseteq J^\star T$, and hence $J^\star T=I^\star T$. By Proposition \ref{prop:starloc:basic}\ref{prop:starloc:basic:equiv}, $\star_T$ is well-defined, and $\star$ is extendable to $T$; in particular, $\lambda_\Theta$ is well-defined.

Moreover, for every $\star\in\insstar(R)$, we have
\begin{equation*}
I^\star=\bigcap_{T\in\Theta}I^\star T=\bigcap_{T\in\Theta}(IT)^{\star_T}
\end{equation*}
using the completeness of $\Theta$ in the first equality and the definition of extension in the second. Thus, $\star=\rho_\Theta\circ\lambda_\Theta(\star)$, i.e., $\rho_\Theta\circ\lambda_\Theta$ is the identity. It follows that $\lambda_\Theta$ is injective and $\rho_\Theta$ is surjective. But $\rho_\Theta$ is injective by Proposition \ref{prop:indip-rhoinj}, so $\lambda_\Theta$ and $\rho_\Theta$ must be bijections.
\end{proof}

The second part of the following corollary is a generalization of \cite[Theorem 2.3]{houston_noeth-starfinite}.
\begin{cor}\label{cor:1dim}
Let $R$ be a one-dimensional integral domain. 
\begin{enumerate}[(a)]
\item $|\insstar(R)|\geq\prod\{|\insstar(R_M)|:M\in\Max(R)\}$;
\item if $R$ is of finite character (for example, if $R$ is Noetherian), then $|\insstar(R)|=\prod\{|\insstar(R_M)|:M\in\Max(R)\}$.
\end{enumerate}
\end{cor}
\begin{proof}
If $M\neq N$ are maximal ideals of $R$, then $R_MR_N=K$, since both $M$ and $N$ have height 1. By Proposition \ref{prop:indip-rhoinj}, there is an injective map from $\insstar(R)$ to the product $\prod\insstar(R_M)$, which in particular implies the first inequality.

If, moreover, $R$ is one-dimensional and of finite character, then $\{R_M\mid M\in\Max(R)\}$ is a Jaffard family, and the claim follows by applying Theorem \ref{teor:star-jaffard}.
\end{proof}

The bijections $\rho_\Theta$ and $\lambda_\Theta$ respect the properties of star operations; see the following Proposition \ref{prop:jaffard-eab} for the eab case.
\begin{teor}\label{teor:jaffard-corresp}
Let $R$ be a domain and $\Theta$ be a Jaffard family on $R$, and let $\star\in\insstar(R)$. Then, $\star$ is of finite type (respectively, semifinite, stable, spectral, Noetherian) if and only if $\star_T$ is of finite type (resp., semifinite, stable, spectral, Noetherian) for every $T\in\Theta$.
\end{teor}
\begin{proof}
By Propositions \ref{prop:starloc:basic}\ref{prop:starloc:basic:ft} and \ref{prop:estensione-prop}, if $\star$ is of finite type, stable, spectral or Noetherian so is $\star_T$. If $\star$ is semifinite, let $I$ be a $\star_T$-closed ideal of $T$, and let $J:=I\cap R$. Then $JT=I$, and $J^\star\subseteq I^{\star_T}\cap R=J$, so that there is a prime ideal $Q\supseteq J$ such that $Q^\star=Q$. For every $U\in\Theta$, $U\neq T$, we have $JU=U$; hence $QU=U$, and thus $QT\neq T$; moreover, since $R$ is flat, $QT$ is prime (Remark \ref{oss:def}(\ref{oss:def:primi})). Therefore, $(QT)^{\star_T}=Q^\star T=QT$ is a proper prime $\star_T$-ideal containing $I$, and $\star_T$ is semifinite.

Let now $\star:=\rho_\Theta(\star^{(T)})$.

If each $\star^{(T)}$ is of finite type, then $\star$ is of finite type by \cite{anderson_examples}.

Suppose each $\star^{(T)}$ is semifinite and $I=I^\star$ is a proper ideal of $R$. Then, $1\notin I$, so there is a $T\in\Theta$ such that $(IT)^{\star^{(T)}}\neq T$, and thus there is a prime ideal $P$ of $T$ containing $IT$ such that $P=P^{\star^{(T)}}$. If $Q:=P\cap R$, then 
\begin{equation*}
Q^\star\subseteq(QT)^{\star^{(T)}}\cap R\subseteq P^{\star^{(T)}}\cap R=Q,
\end{equation*}
so that $Q$ is a $\star$-prime ideal of $R$ containing $I$.

If each $\star^{(T)}$ is stable, then, given ideal $I,J$ of $R$, we have
\begin{equation*}
(I\cap J)^\star=\bigcap_{T\in\Theta}((I\cap J)T)^{\star^{(T)}}=\bigcap_{T\in\Theta}(IT)^{\star^{(T)}}\cap \bigcap_{T\in\Theta}(JT)^{\star{(T)}}=I^\star\cap J^\star.
\end{equation*}
Hence, $\star$ is stable. The case of spectral star operation follows since $\star$ is spectral if and only if it is stable and semifinite \cite[Theorem 4]{anderson_overrings_1988}.

Suppose now $\star^{(T)}$ is Noetherian for every $T\in\Theta$ and let $\{I_\alpha:\alpha\in A\}$ be an ascending chain of $\star$-ideals. If $I_\alpha=(0)$ for every $\alpha$ we are done. Otherwise, there is a $\overline{\alpha}$ such that $I_{\overline{\alpha}}\neq(0)$, and thus $I_{\overline{\alpha}}$ (and, consequently, every $I_\alpha$ for $\alpha>\overline{\alpha}$) survives in only a finite number of elements of $\Theta$, say $T_1,\ldots,T_n$. For each $i\in\{1,\ldots,n\}$, the set $\{I_\alpha T_i\}$ is an ascending chain of $\star^{(T_i)}$-ideals, and thus there is a $\alpha_i$ such that $I_\alpha T_i=I_{\alpha_i}T_i$ for every $\alpha\geq\alpha_i$.

Let thus $\widetilde{\alpha}:=\max\{\overline{\alpha},\alpha_i: 1\leq i\leq n\}$. For every $\beta\geq\widetilde{\alpha}$, we have $I_\beta T_i=I_{\alpha_i}T_i=I_{\widetilde{\alpha}}T_i$, while, if $T\neq T_i$ for every $i$, then $I_\beta T=T=I_{\widetilde{\alpha}}T$ since $\beta\geq\overline{\alpha}$. Therefore, $I_\beta=\bigcap_{T\in\Theta}I_\beta T=\bigcap_{T\in\Theta}I_{\widetilde{\alpha}}T=I_{\widetilde{\alpha}}$ and the chain $\{I_\alpha\}$ stabilizes.
\end{proof}

\begin{cor}\label{cor:trasfnoeth}
Let $R$ be a domain and $\Theta$ be a Jaffard family on $R$. If every $T\in\Theta$ is Noetherian, so is $R$.
\end{cor}
\begin{proof}
A domain $A$ is Noetherian if and only if the identity star operation $d^{(A)}$ is Noetherian. If every $T\in\Theta$ is Noetherian, each $d_T$ is a Noetherian star operation, and thus (by Theorem \ref{teor:jaffard-corresp}) $\rho_\Theta(d_T)$ is Noetherian. However, by Theorem \ref{teor:star-jaffard}, $\rho_\Theta(d_T)=d_R$, and thus $R$ is a Noetherian domain.
\end{proof}

\begin{lemma}\label{lemma:IcapRJcapR}
Let $R$ be an integral domain and let $T$ be a Jaffard overring of $R$. For all nonzero integral ideals $I,J$ of $T$,
\begin{equation*}
(I\cap R)(J\cap R)=IJ\cap R.
\end{equation*}
\end{lemma}
\begin{proof}
Let $\Theta$ be a Jaffard family containing $T$. Since $\Theta$ is complete, it is enough to show that they are equal when localized on every $U\in\Theta$. We have
\begin{equation*}
(I\cap R)(J\cap R)U=(IU\cap U)(JU\cap U)=\begin{cases}
IJ & \text{~if~}U=T\\
U & \text{~if~}U\neq T
\end{cases}
\end{equation*}
while
\begin{equation*}
(IJ\cap R)U=IJU\cap U=\begin{cases}
IJ & \text{~if~}U=T\\
U & \text{~if~}U\neq T
\end{cases}
\end{equation*}
and thus $(I\cap R)(J\cap R)=IJ\cap R$.
\end{proof}

\begin{lemma}\label{lemma:IcapR}
Let $R$ be an integral domain, $T$ a Jaffard overring of $R$, and let $I$ be a finitely generated integral ideal of $T$. Then, $I\cap R$ is finitely generated (over $R$).
\end{lemma}
\begin{proof}
Let $S:=\ortog{\Theta}(T)$, where $\Theta$ is a Jaffard family to which $T$ belongs. Then, by Proposition \ref{prop:jaffard:basic}, $(I\cap R)S=IS\cap S=ITS\cap S=S$, and thus there are $i_1,\ldots,i_n\in I\cap R$, $s_1,\ldots,s_n\in S$ such that $1=i_1s_1+\cdots+i_ns_n$; let $I_0:=(i_1,\ldots,i_n)$.

Let $x_1,\ldots,x_m$ be the generators of $I$ in $T$. Since $(I\cap R)T=IT=I$, for every $x_i$ there are $j_{1i},\ldots,j_{n_ii}\in I\cap R$, $t_{1i},\ldots,t_{n_ii}\in T$ such that $x_i=j_{1i}t_{1i}+\cdots+j_{n_ii}t_{n_ii}$; let $I_i:=(j_{1i},\ldots,j_{n_ii})$. Then, $J:=I_0+I_1+\cdots+I_n$ is a finitely generated ideal contained in $I\cap R$ (since it is generated by elements of $I\cap R$) such that $(I\cap R)T\subseteq JT$ and $(I\cap R)S\subseteq JS$; thus, $I\cap R\subseteq J$. Therefore, $I\cap R=J$ is finitely generated, as claimed.
\end{proof}

\begin{prop}\label{prop:jaffard-eab}
Let $R$ be an integral domain and let $\Theta$ be a Jaffard family on $R$. A $\star\in\insstar(R)$ is eab (resp., ab) if and only if $\star_T$ is eab (resp., ab) for every $T\in\Theta$.
\end{prop}
\begin{proof}
$(\Longrightarrow)$. Suppose $(IJ)^{\star_T}\subseteq(IL)^{\star_T}$ for some finitely generated ideals $I,J,L$ of $T$ (which we can suppose contained in $T$). Since 
\begin{equation*}
(IJ\cap R)^\star T=((IJ\cap R)T)^{\star_T}=(IJ)^{\star_T}
\end{equation*}
(and the same happens for $IL$), we have $(IJ\cap R)^\star T\subseteq (IL\cap R)^\star T$, and so 
\begin{equation*}
(IJ\cap R)^\star T\cap R\subseteq (IL\cap R)^\star T\cap R.
\end{equation*}
However, both $IJ\cap R$ and $IL\cap R$ survive (among the ideals of $\Theta$) only in $T$, so that 
\begin{equation*}
(IJ\cap R)^\star T\cap R=(IJ\cap R)^\star=((I\cap R)(J\cap R))^\star
\end{equation*}
by Lemma \ref{lemma:IcapRJcapR}, and thus 
\begin{equation*}
((I\cap R)(J\cap R))^\star\subseteq ((I\cap R)(L\cap R))^\star.
\end{equation*}
Since $I$ is finitely generated, by Lemma \ref{lemma:IcapR} so is $I\cap R$; the same happens for $J\cap R$ and $L\cap R$. Hence, since $\star$ is eab, $(J\cap R)^\star\subseteq(L\cap R)^\star$, and thus 
\begin{equation*}
J^{\star_T}=(J\cap R)^\star T\subseteq(L\cap R)^\star T=L^{\star_T}.
\end{equation*}
Hence, $\star_T$ is eab.

$(\Longleftarrow)$. Suppose $(IJ)^\star\subseteq(IL)^\star$. Then, $(IJ)^\star T\subseteq(IL)^\star T$, i.e., $(IJT)^{\star_T}\subseteq(ILT)^{\star_T}$ for every $T\in\Theta$. Since $\star_T$ is eab, this implies that $(JT)^{\star_T}\subseteq(LT)^{\star_T}$ for every $T\in\Theta$; since $H^\star=\bigcap_{T\in\Theta}(HT)^{\star_T}$, it follows that $J^\star\subseteq L^\star$, and $\star$ is eab.

The same reasoning applies for the ab case.
\end{proof}

Following \cite{hhp_m-canonical}, we say that an ideal $A$ is $m$-canonical if $I=(A:(A:I))$ for every fractional ideal $I$ of $R$. The following proposition can be seen as a generalization of \cite[Theorem 6.7]{hhp_m-canonical} to domains that are not necessarily integrally closed.
\begin{prop}\label{prop:starloc:mcan}
Let $R$ be a domain. Then $R$ admits an $m$-canonical ideal if an only if $R$ is $h$-local, $R_M$ admits an $m$-canonical ideal for every $M\in\Max(R)$ and $|\insstar(R_M)|\neq 1$ for only a finite number of maximal ideals of $M$.
\end{prop}
\begin{proof}
Suppose $A$ is $m$-canonical. Then $R$ is $h$-local by \cite[Proposition 2.4]{hhp_m-canonical}; moreover, if $I$ is a $R_M$-fractional ideal, then $I=JR_M$ for some $R$-fractional ideal, and thus
\begin{equation*}
(AR_M:(AR_M:I))=(AR_M:(AR_M:JR_M))=\eql
=(AR_M:(A:J)R_M)=(A:(A:J))R_M=JR_M=I
\end{equation*}
applying Lemma \ref{lemma:Tcolon} (which is applicable since $R$ $h$-local implies that $R_M$ is a Jaffard overring of $R$). If $AR_M=R_M$, it follows that $R_M$ is an $m$-canonical ideal for $R_M$, and thus that the $v$-operation on $R_M$ is the identity, or equivalently that $|\insstar(R_M)|=1$; hence, if $|\insstar(R_M)|\neq 1$ then $AR_M\neq R_M$. But this can happen only for a finite number of $M$, since $R$ is $h$-local and thus of finite character.

Conversely, suppose that the three hypotheses hold. For every $M\in\Max(R)$, let $J_M$ be an $m$-canonical ideal of $R_M$, and define
\begin{equation*}
I_M:=\begin{cases}
R_M & \text{~if~}|\insstar(R_M)|=1\\
J_M & \text{~if~}|\insstar(R_M)|>1
\end{cases}
\end{equation*}
Note that, if $|\insstar(R_M)|=1$, then $R_M$ is $m$-canonical for $R_M$, and thus $I_M$ is $m$-canonical for every $M$.

The ideal $J:=\bigcap_{P\in\Max(R)}I_P$ of $R$ is nonzero, and by Lemma \ref{lemma:intersez-ritorno} $JR_M=I_M$ for every maximal ideal $M$. If $L$ is an ideal of $R$ then, for every maximal ideal $M$,
\begin{equation*}
(J:(J:L))R_M=(JR_M:(JR_M:LR_M))=(I_M:(I_M:LR_M))=LR_M,
\end{equation*}
so that
\begin{equation*}
(J:(J:L))=\bigcap_{M\in\Max(R)}(J:(J:L))R_M=\bigcap_{M\in\Max(R)}LR_M=L.
\end{equation*}
Therefore, $J$ is an $m$-canonical ideal of $R$.
\end{proof}

\begin{oss}
The results in Sections \ref{sect:extension} and \ref{sect:jaff-star} can be generalized in two different directions.

On the one hand, we can consider, instead of star operations, other classes of closure operations, for example semiprime or semistar operations. In both cases, the definitions of extendability and the results in Section \ref{sect:extension} carry over without modifications, noting that the equalities $(I^c:J^c)=(I^c:J)$ and $(I^\star:J^\star)=(I^\star:J)$ holds when $c$ and $\star$ are, respectively, a semiprime or a semistar operation.

However, the behaviour of these two classes differs when we come to Jaffard families. In one case there is no problem: with the obvious modifications, all result of Section \ref{sect:jaff-star} hold for the set $\inssemiprime(R)$ of semiprime operations. For example, this means that we can analyze the structure of the semiprime operation on a Dedekind domain $D$ almost directly from the structure of $\inssemiprime(V)$, for $V$ a discrete valuation ring, shortening the analysis done in \cite[Section 3]{vassilev_structure_2009}.

The case of semistar operations is much more delicate: indeed, the result corresponding to  Theorem \ref{teor:star-jaffard} is \emph{not} true for $\inssemistar(R)$, meaning that a semistar operation on $R$ may not be extendable to a Jaffard overring $T$ of $R$. For example, let $\star$ be the semistar operation defined by
\begin{equation*}
I^\star=\begin{cases}
I & \text{if~}I\in\insfracid(R)\\
K & \text{otherwise}.
\end{cases}
\end{equation*}
If $T\neq R$ is a Jaffard overring of $R$, then it is not a fractional ideal of $R$ (for otherwise $T\cdot\ortog{\Theta}(T)=K$ would imply $\ortog{\Theta}(T)=K$); however, we have $RT=TT$, while
\begin{equation*}
R^\star T=T\neq K=T^\star T.
\end{equation*}
Hence, $\star$ is not extendable to $T$. The exact point in which the proof of Theorem \ref{teor:star-jaffard} fails is the possibility of using Lemma \ref{lemma:Tcolon}, because the equality $IT=JT$ does not imply that $(I:J)\neq(0)$. However, if we restrict to finite-type semistar operations, the analogue of Theorem \ref{teor:star-jaffard} does hold: indeed, a proof analogous to the one of Proposition \ref{prop:starloc:basic}\ref{prop:starloc:basic:ft} shows that finite-type operations are extendable, and thus the proof of Theorem \ref{teor:star-jaffard} continues without problems.

A second way of generalizing these results is by considering, beyond the order structure, also a \emph{topological} structure on $\insstar(R)$: mimicking the definition of the Zariski topology on $\inssemistar(R)$ given in \cite{topological-cons}, we can define a topology on $\insstar(R)$ by declaring open the sets of the form
\begin{equation*}
V_I:=\{\star\in\insstar(R)\mid 1\in I^\star\},
\end{equation*}
as $I$ ranges among the fractional ideals of $R$. In particular, Theorem \ref{teor:star-jaffard} can be interpreted at the topological level: if $\Theta$ is a Jaffard family of $R$, then $\lambda_\Theta$ and $\rho_\Theta$ are homeomorphisms between $\insstar(R)$ and the space $\prod_{T\in\Theta}\insstar(T)$ endowed with the product topology.
\end{oss}

\section{Application to Pr\"ufer domains}\label{sect:starloc:prufer}
Theorem \ref{teor:star-jaffard} allows one to split the study of the set $\insstar(R)$ of star operations on $R$ into the study of the sets $\insstar(T)$, as $T$ ranges among the members of a Jaffard family $\Theta$. Obviously, this result isn't quite useful if we don't know how to find Jaffard families, or if studying $\insstar(T)$ is as complex as studying $\insstar(R)$. The purpose of this section is to show that, in the case of (some classes of) Pr\"ufer domains, we can resolve the first question, and we can at least make some progress on the second, proving more explicit results on $\insstar(R)$. We shall employ a method similar to the one used in \cite[Sections 3-5]{hmp_finite}

Let now $R$ be a Pr\"ufer domain with quotient field $K$. We say that two maximal ideals $M,N$ are \emph{dependent} if $R_MR_N\neq K$, or equivalently if $M\cap N$ contains a nonzero prime ideal. Since the spectrum of $R$ is a tree, being dependent is an equivalence relation; we indicate the equivalence classes by $\Delta_\lambda$, as $\lambda$ ranges in $\Lambda$. We also define $T_\lambda:=\bigcap\{R_P\mid P\in\Delta_\lambda\}$. We call the set $\Theta:=\{T_\lambda\mid \lambda\in\Lambda\}$ the \emph{standard decomposition} of $R$.

\begin{lemma}\label{lemma:dimfin}
Let $R$ be a finite-dimensional Pr\"ufer domain. Then, $\Delta\subseteq\Max(R)$ is an equivalence class with respect to dependence if and only if $\Delta=V(P)\cap\Max(R)$ for some height-one prime $P$ of $R$.
\end{lemma}
\begin{proof}
Suppose $\Delta=V(P)\cap\Max(R)$. If $M,N\in\Delta$, then $P\subseteq M\cap N$; conversely, since $P$ has height 1, $M\in\Delta$ and $Q\subseteq M\cap N$ imply that $P\subseteq Q$ (since the spectrum of $R$ is a tree).

On the other hand, suppose $\Delta=\Delta_\lambda$ for some $\lambda$, and let $M,N\in\Delta$. Since $\Spec(R)$ is a tree and $\dim(R)<\infty$, both $M$ and $N$ contain a unique height-one prime, respectively (say) $P_M$ and $P_N$; if $P_M\neq P_N$, then $M\cap N$ cannot contain a nonzero prime, and thus $M$ and $N$ are not dependent, against the hypothesis $M,N\in\Delta$. Therefore, the height-1 prime contained in the members of $\Delta$ is unique, and $\Delta=V(P)\cap\Max(R)$.
\end{proof}

\begin{prop}\label{prop:Tlambda}
Let $R$ be a Pr\"ufer domain, and suppose that
\begin{enumerate}[(a)]
\item\label{prop:Tlambda:noeth} $\Max(R)$ is a Noetherian space; or
\item\label{prop:Tlambda:semiloc} $R$ is semilocal.
\end{enumerate}
Then, the standard decomposition $\Theta$ of $R$ is a Jaffard family of $R$.
\end{prop}
\begin{proof}
Since $R$ is Pr\"ufer, every overring of $R$ is flat \cite[Theorem 1.1.1]{fontana_libro}, and this in particular applies to the $T\in\Theta$.

We claim that, under both hypotheses, if $T=T_\lambda\in\Theta$, then $\Spec(T)=\{PT\mid P\subseteq M\text{~for some~}M\in\Delta_\lambda\}$. Indeed, in both cases every $\Delta_\lambda$ is compact: if $\Max(R)$ is Noetherian this is immediate, while if $R$ is semilocal they are finite and thus compact. Hence, the semistar operation $\star_\Delta$ is of finite type \cite[Corollary 4.6]{localizing-semistar}, and $R^{\star_\Delta}=T$; since the unique finite-type (semi)star operation on a Pr\"ufer domain is the identity (since all finitely generated ideals are invertible), it follows that $\star_\Delta$ is just the map $I\mapsto IT$, and thus $QT=T$ if $Q$ is not contained in any $M\in\Delta$. Therefore, no prime ideal $P$ of $R$ survives in two different members of $\Theta$; thus, $PT_\lambda T_\mu=T_\lambda T_\mu$ if $\lambda\neq\mu$ are in $\Lambda$. Hence, $T_\lambda T_\mu=K$.

We need to show that $\Theta$ is locally finite. If $R$ is semilocal then $\Theta$ is finite, and in particular locally finite; suppose $\Max(R)$ is Noetherian. For every $x\in R$, $x\neq 0$, the ideal $xR$ has only a finite number of minimal primes (this follows, for example, from the proof of \cite[Chapter 4, Corollary 3, p.102]{bourbaki_ac} or \cite[Chapter 6, Exercises 5 and 7]{atiyah}); in particular, since each prime survives in only one $T\in\Theta$, the family $\Theta$ is of finite character.

Hence, in both case $\Theta$ is a Jaffard family  by Proposition \ref{prop:caratt-jaffard}.
\end{proof}

\begin{oss}\label{oss:prufjaff}
~\\
\begin{enumerate}
\item If $R$ is a Pr\"ufer domain that is both of finite character and finite-dimensional, then $\Spec(R)$ (and so $\Max(R)$) is Noetherian. Indeed, if $I$ is a nonzero radical ideal of $R$, then $V(I)$ is finite, and thus every ascending chain of radical ideals must stop; by \cite[Chapter 6, Exercise 5]{atiyah}, this implies Noetherianity.
\item\label{oss:prufjaff:finerJaff} 
The standard decomposition $\Theta$ of $R$ is the ``finest'' Jaffard family of $R$, in the sense that the partition of $\Max(R)$ determined by $\Theta$ (see Remark \ref{oss:Matlis}) is the finer partition that can be induced by a Jaffard family; this follows exactly from the definition of the dependence relation.
\item\label{oss:prufjaff:nonminimalbranch} In general, the standard decomposition of $R$ need not be a Jaffard family of $R$. For example, let $R$ be an almost Dedekind domain which is not Dedekind. Since $R$ is one-dimensional, no two maximal ideals are dependent, and thus each $T_\lambda$ has the form $R_M$ for some maximal ideal $M$. However, $\Theta$ is not a Jaffard family, since it is not locally finite (if it were, $R$ would be a Dedekind domain). Indeed, Example \ref{ex:ad} shows that not every star operation is extendable to every $R_M$.
\end{enumerate}
\end{oss}

\subsection{Cutting the branch}
Let $R$ be a finite-dimensional Pr\"ufer domain whose standard decomposition $\Theta$ is a Jaffard family. By Lemma \ref{lemma:dimfin}, every $T\in\Theta$ will have a nonzero prime ideal $P$ contained in all its maximal ideals; moreover, by Remark \ref{oss:prufjaff}(\ref{oss:prufjaff:finerJaff}), $T$ does not admit a further decomposition. On the other hand, it may be possible that $T/P$ has a nontrivial standard decomposition that is still a Jaffard family; thus, if we could relate $\insstar(T)$ with $\insstar(T/P)$, we could (in principle) simplify the study of $\insstar(T)$.

\begin{lemma}\label{lemma:prufer-jac}
Let $R$ be a Pr\"ufer domain whose Jacobson radical $\Jac(R)$ contains a nonzero prime ideal. Then, there is a prime ideal $Q\subseteq\Jac(R)$ such that $\Jac(R/Q)$ does not contain nonzero prime ideals.
\end{lemma}
\begin{proof}
Let $\Delta:=\{P\in\Spec(R),P\subseteq\Jac(R)\}$. By hypothesis, $\Delta$ contains nonzero prime ideals. Let $Q:=\bigcup_{P\in\Delta}P$.

Since $R$ is treed, $\Delta$ is a chain; hence, $Q$ is itself a prime ideal, and it is contained in every maximal ideal of $R$. Suppose $\Jac(R/Q)$ contains a nonzero prime ideal $\overline{Q}$. Then, $\overline{Q}=Q'/P$ for some prime ideal $Q'$ of $R$, and $Q'$ is contained in every maximal ideal of $R$. It follows that $Q\subsetneq Q'\subseteq\Jac(R)$, against the construction of $Q$.
\end{proof}

Suppose now that $R$ is a Pr\"ufer domain with quotient field $K$, and suppose there is a nonzero prime ideal $P$ contained in every maximal ideal of $R$. Then, we have a quotient map $\phi:R_P\longrightarrow R_P/PR_P=k$ that, for every star operation $\star$ on $R$, induces a semistar operation $\star_\phi$ on $D:=R/P$ defined by
\begin{equation*}
I^{\star_\phi}:=\phi\left(\phi^{-1}(I)^\star\right),
\end{equation*}
such that $D^{\star_\phi}=D$. Conversely, if $\sharp$ is a star operation on $D$, then we can construct a star operation $\sharp^\phi$ on $R$: indeed, if $I$ is a fractional ideal of $R$, then $I$ is either divisorial (and so we define $I^{\sharp^\phi}:=I$) or there is an $\alpha\in K$ such that $R\subseteq \alpha I\subseteq R_P$ \cite[Proposition 2.2(5)]{hmp_finite}: in the latter case, we define
\begin{equation*}
I^{\sharp^\phi}:=\alpha^{-1}\phi^{-1}\left(\phi(\alpha I)^\sharp\right).
\end{equation*}

\begin{prop}\label{prop:star-semistar}
Let $R,P,D,\phi$ as above. Then, the maps
\begin{equation*}
\begin{aligned}
\insstar(R) & \longrightarrow \inssmstar(R/P)\\
\star & \longmapsto \star_\phi
\end{aligned}
\quad\text{~and~}\quad
\begin{aligned}
\inssmstar(R/P) & \longrightarrow \insstar(R)\\
\star & \longmapsto \star^\phi
\end{aligned}
\end{equation*}
are well-defined order-preserving bijections.
\end{prop}
\begin{proof}
The fact that they are well-defined and bijections follow from \cite[Lemmas 2.3 and 2.4]{hmp_finite}; the fact that they are order-preserving is immediate from the definitions.
\end{proof}

\subsection{$h$-local Pr\"ufer domains}\label{sect:hloc-prufer}
If $R$ is both a Pr\"ufer domain and a $h$-local domain, then its standard decomposition $\Theta:=\{R_M\mid M\in\Max(R)\}$ is composed by valuation domains, and star operations behave particularly well. We start by re-proving \cite[Theorem 3.1]{twostar} using our general theory.
\begin{prop}\label{prop:hloc-prufer}
Let $R$ be an $h$-local Pr\"ufer domain, and let $\mathcal{M}$ be the set of nondivisorial maximal ideals of $R$. Then, $|\insstar(R)|=2^{|\mathcal{M}|}$.
\end{prop}
\begin{proof}
By Theorem \ref{teor:star-jaffard}, there is an order-preserving bijection between $\insstar(R)$ and $\prod\{\insstar(R_M)\mid M\in\Max(R)\}$, and a maximal ideal $M$ is divisorial (in $R$) if and only if $MR_M$ is divisorial (in $R_M$). Since $R_M$ is a valuation domain, $|\insstar(R_M)|$ is equal to 1 if $MR_M$ is divisorial, and to 2 if $MR_M$ is not; the claim follows.
\end{proof}

It is noted in the proof of \cite[Theorem 3.10]{olberding_globalizing} that, if $R$ is an $h$-local Pr\"ufer domain and $I,J$ are divisorial ideals of $R$, then $I+J$ is also divisorial. We can extend this result to arbitrary star operations; we shall see a similar result in Proposition \ref{prop:pruf-somma-invt}.
\begin{prop}\label{prop:hloc-pruf-somma}
Let $R$ be an $h$-local Pr\"ufer domain, let $\star\in\insstar(R)$ and let $I,J$ be $\star$-closed ideals. Then, $I+J$ is $\star$-closed.
\end{prop}
\begin{proof}
Since $R$ is $h$-local, $I+J$ is $\star$-closed if and only if $(I+J)R_M$ is $\star_M$-closed for every $M\in\Max(R)$. However, since $R_M$ is a valuation domain, either $IR_M\subseteq JR_M$ or $JR_M\subseteq IR_M$; hence, $(I+J)R_M=IR_M+JR_M$ is equal either to $IR_M$ or to $JR_M$, both of which are $\star_M$-closed.
\end{proof}

This result does not hold if we drop the hypothesis that $R$ is $h$-local: in fact, let $R=\insZ+X\insQ[[X]]$ and let $R_p:=\insZ[1/p]+X\insQ[[X]]$ for each prime number $p$. Consider the star operation 
\begin{equation*}
\star:I\mapsto(R:(R:I))\cap(R_2:(R_2:I))\cap(R_3:(R_3:I)).
\end{equation*}
Then, $R_2$ and $R_3$ are $\star$-closed; we claim that $R_2+R_3$ is not. Indeed, if $T$ is equal to $R$, $R_2$ or $R_3$, then $(T:(R_2+R_3))=X\insQ[[X]]$, and thus $(R_2+R_3)^\star=\insQ[[X]]$; however, $R_2+R_3=(\insZ[1/2]+\insZ[1/3])+X\insQ[[X]]$ does not contain rationals with denominator not divisible by 2 or 3 (for example, $1/5\notin R_2+R_3$), and thus $R_2+R_3\neq\insQ[[X]]$.

The following can be seen as a sort of converse to Proposition \ref{prop:hloc-pruf-somma}.
\begin{prop}\label{prop:sommaintersez-prufer}
Let $R$ be a Pr\"ufer domain and suppose that $R$ is either:
\begin{enumerate}[(a)]
\item semilocal; or
\item locally finite and finite-dimensional.
\end{enumerate}
Then, the following are equivalent:
\begin{enumerate}[(i)]
\item $R$ is $h$-local;
\item for every $\star\in\insstar(R)$, $I\in\insfracid(R)\setminus\insstarid^\star(R)$ and $J\in\insfracid(R)$, at least one of $I\cap J$ and $I+J$ is not $\star$-closed;
\item for every $I\in\insfracid(R)\setminus\insstarid^v(R)$ and $J\in\insfracid(R)$, at least one of $I\cap J$ and $I+J$ is not divisorial.
\end{enumerate}
\end{prop}
\begin{proof}
(i $\Longrightarrow$ ii) For every $M\in\Max(R)$, $(I+J)R_M=IR_M+JR_M=\max\{IR_M,JR_M\}$, while $(I\cap J)R_M=IR_M\cap JR_M=\min\{IR_M,JR_M\}$. Since $I$ is not $\star$-closed, and $\{R_M\mid M\in\Max(R)\}$ is a Jaffard family of $R$, there is a maximal ideal $N$ such that $IR_N$ is not $\star_N$-closed; however, at least one of $(I+J)R_N$ and $(I\cap J)R_N$ is equal to $IR_N$, and thus at least one is not $\star_N$-closed. Therefore, at least one between $I+J$ and $I\cap J$ is not $\star$-closed.

(ii $\Longrightarrow$ iii) is obvious.

(iii $\Longrightarrow$ i) Consider the standard decomposition $\Theta$ of $R$; then, (iii) holds for every member of $\Theta$ but, if $R$ is not $h$-local, there must be a $T\in\Theta$ that is not local. By Lemma \ref{lemma:prufer-jac}, there is a prime ideal $P$ of $T$ such that $\Jac(T/P)$ does not contain nonzero primes. Let $\Lambda$ be the standard decomposition of $D:=T/P$, let $Z\in\Lambda$ and define $Z':=\bigcap_{W\in\Lambda\setminus\{Z\}}W=\ortog{\Lambda}(Z)$. We have $Z\cap Z'=D$ and, for every maximal ideal $M$ of $D$, either $ZD_M=K$ or $Z'D_M=K$. Therefore, $Z+Z'=\bigcap_{M\in\Max(T)}(Z+Z')D_M=K$. 

By Proposition \ref{prop:star-semistar}, the $v$-operation on $T$ correspond to a (semi)star operation on $D$ such that $A^\star=K$ if $A$ is not a fractional ideal of $D$; therefore, both $\phi^{-1}(Z)$ and $\phi^{-1}(Z')$ are not divisorial, but both $\phi^{-1}(Z\cap Z')=T$ and $\phi^{-1}(Z+Z')=T_P$ are (where $\phi:T\longrightarrow D$ is the quotient map). This is a contradiction, and $R$ must be $h$-local.
\end{proof}

\subsection{Stability}
Recall that a star operation $\star$ is \emph{stable} if it distributes over finite intersections, i.e., if $(I\cap J)^\star=I^\star\cap J^\star$. In this section, we study stable operations on Pr\"ufer domains; we start with an analogue of Proposition \ref{prop:star-semistar}.
\begin{prop}\label{prop:star-semistar_stab}
Preserve the notation and the hypotheses of Proposition \ref{prop:star-semistar}. There is a bijection between $\insstarstab(R)$ and $\insstarstab(R/P)$.
\end{prop}
\begin{proof}
We first show that the bijections of Proposition \ref{prop:star-semistar}  become bijections on the subsets of stable operations; let thus $\star$ be a semistar operation in the first set and $\sharp$ be the corresponding operation on $\inssmstar(R/P)$. Let $\phi:R\longrightarrow R/P$ be the quotient map.

Suppose that $\star$ is stable and let $I,J\in\inssubmod(R/P)$. Then, since $\phi$ is a bijection between the ideal comprised between $P$ and $R_P$ and $\inssubmod(R/P)$,
\begin{equation*}
\begin{array}{rcl}
(I\cap J)^\sharp & = & \phi\left[\phi^{-1}(I\cap J)^\star\right]= \phi\left[\left(\phi^{-1}(I)\cap\phi^{-1}(J)\right)^\star\right]=\\
& = & \phi\left[\phi^{-1}(I)^\star\cap\phi^{-1}(J)^\star\right]= \phi\left(\phi^{-1}(I)^\star\right)\cap\phi\left(\phi^{-1}(J)^\star\right)=\\
& = & I^\sharp\cap J^\sharp.
\end{array}
\end{equation*}
Therefore, $\sharp$ is stable.

Conversely, suppose $\sharp$ is stable and let $I,J\in\insfracid(R)$. If $I$ and $J$ are divisorial, so is $I\cap J$; hence, $(I\cap J)^\star=I\cap J=I^\star\cap J^\star$. Suppose (without loss of generality) that $I\neq I^v$. Then, there is an $\alpha$ such that $P\subseteq \alpha I\subseteq R_P$. Moreover, since $R$ is Pr\"ufer and $P$ is contained in every maximal ideal of $R$, every fractional ideal must be comparable with both $P$ and $R_P$: more precisely, if $\valut$ is the valuation relative to $R_P$, and $L$ is an ideal, then either $\inf\valut(L)=0$ (so that $P\subseteq L\subseteq R_P$), $\inf\valut(L)$ exist and has a sign (if positive, $L\subseteq P$, if negative, $R_P\subseteq L$) or $\inf\valut(L)$ has no infimum (so that if $\valut(L)$ contains negative values then $R_P\subseteq L$, while $L\subseteq P$ in the other case). Therefore, we can distinguish three cases:
\begin{itemize}
\item $\alpha J\subseteq P$: then, $\alpha J\subseteq \alpha I$, and thus $(I\cap J)^\star=J^\star=I^\star\cap J^\star$;
\item $R_P\subseteq\alpha J$: then, $\alpha I\subseteq\alpha J$, and thus $(I\cap J)^\star=I^\star=I^\star\cap J^\star$;

\item $P\subseteq\alpha J\subseteq R_P$. Let $I_0:=\alpha I$ and $J_0:=\alpha J$. Then,
\begin{equation*}
\begin{array}{rcl}
(I_0\cap J_0)^\star & = &\phi^{-1}\left(\phi(I_0\cap J_0)^\sharp\right)= \phi^{-1}\left(\phi(I_0)^\sharp\cap\phi(J_0)^\sharp\right)=\\
& = & \phi^{-1}(\phi(I_0)^\sharp)\cap\phi^{-1}(\phi(J_0)^\sharp)=I_0^\star\cap J_0^\star.
\end{array}
\end{equation*}
Hence,
\begin{equation*}
\begin{array}{rcl}
(I\cap J)^\star & = & \alpha^{-1}(\alpha(I\cap J)^\star)=\alpha^{-1}(I_0\cap J_0)^\star=\\
& = & \alpha^{-1}(I_0^\star\cap J_0^\star)=\alpha^{-1}I_0^\star\cap\alpha^{-1}J_0^\star=I^\star\cap J^\star.
\end{array}
\end{equation*}
\end{itemize}
In all cases, $\star$ distributes over finite intersection, and thus $\star$ is stable.

Therefore, there is an order-preserving bijection between $\insstarstab(R)$ and $\inssmstar_{\mathrm{st}}(R/P)$. However, for every domain $D$, the restriction map $\inssmstar_{\mathrm{st}}(D)\longrightarrow\insstarstab(D)$ is a bijection (see \cite[Discussion after Proposition 3.10]{surveygraz} or \cite[Proposition 3.4]{spettrali-eab}), and thus $\insstarstab(R)$ corresponds bijectively with $\insstarstab(R/P)$. The claim follows.
\end{proof}

We say that a star (or semistar) operation $\star$ \emph{distributes over arbitrary intersections} if, whenever $\{I_\alpha\}_{\alpha\in A}$ is a family of ideals with nonzero intersection, we have $\left(\bigcap_{\alpha\in A}I_\alpha\right)^\star=\bigcap_{\alpha\in A}I_\alpha^\star$.

\begin{lemma}\label{lemma:intersez-valuation}
If $V$ is a valuation domain, the $v$-operation distributes over arbitrary intersections.
\end{lemma}
\begin{proof}
Let $\mathcal{A}:=\{I_\alpha\}_{\alpha\in A}$ be a family of ideals of $V$ with nonzero intersection. If $\mathcal{A}$ has a minimum $I_{\overline{\alpha}}$, then $I_{\overline{\alpha}}^v\subseteq I_\beta^v$ for every $\beta\in A$, and thus $\left(\bigcap_{\alpha\in A}I_\alpha\right)^v=I_{\overline{\alpha}}^v=\bigcap_{\alpha\in A}I_\alpha^v$.

Suppose $\mathcal{A}$ does not have a minimum: since $\left(\bigcap_{\alpha\in A}I_\alpha\right)^v\subseteq I_\alpha^v$ for every $\alpha\in A$, we have $\left(\bigcap_{\alpha\in A}I_\alpha\right)^v\subseteq\bigcap_{\alpha\in A}I_\alpha^v$.

Let $x\in\bigcap_{\alpha\in A}I_\alpha^v$: if $x\in\bigcap_{\alpha\in A}I_\alpha$ then $x\in\left(\bigcap_{\alpha\in A}I_\alpha\right)^v$. On the other hand, if $x\notin\bigcap_{\alpha\in A}I_\alpha$, then there is an $\overline{\alpha}$ such that $x\in I_{\overline{\alpha}}^v\setminus I_{\overline{\alpha}}$, i.e., $\valut(x)=\inf \valut(I_{\overline{\alpha}})$ (where $\valut$ is the valuation associated to $V$ and $\valut(J):=\{\valut(j)\mid j\in J\}$). However, since $\mathcal{A}$ has no minimum, there are $\beta,\gamma\in A$ such that $I_\alpha\supsetneq I_\beta\supsetneq I_\gamma$; in particular, $\valut(x)>\inf \valut(I_\gamma)$, and thus $x\notin I_\gamma^v$, which is absurd. Therefore, $x\in\bigcap_{\alpha\in A}I_\alpha$.
\end{proof}

The following proposition may also be proved, in a slightly more generalized setting, using a different, more direct, approach; see \cite{stable_prufer}.
\begin{prop}\label{prop:prufer-stab}
Let $R$ be a Pr\"ufer domain and suppose that $R$ is either:
\begin{enumerate}[(a)]
\item semilocal; or
\item locally finite and finite-dimensional.
\end{enumerate}
Then, every stable star operation $\star$ on $R$ is in the form
\begin{equation}\label{eq:stablepruf}
I\mapsto\bigcap_{P\in\Max(R)}(IR_P)^{\star^{(P)}},
\end{equation}
where each $\star^{(P)}\in\insstar(R_P)$. In particular, $\insstarstab(R)$ is order-isomorphic to $\prod\{\insstar(R_P)\mid P\in\Max(R)\}$.
\end{prop}
\begin{proof}
For any ring $A$, let $\mathcal{M}_A$ be the set of maximal ideals of $A$ that are not divisorial.

Suppose first that $R$ is semilocal, and let $\Delta$ be the set of star operations defined as in \eqref{eq:stablepruf}. By Lemma \ref{lemma:intersez-valuation}, every star operation in $\Delta$ is stable; moreover, a maximal ideal $P$ is $\star$-closed if and only if $\star^{(P)}$ is the identity, and thus $|\Delta|=2^{|\mathcal{M}_R|}$. Since $\insstar(R)$ is finite \cite[Theorem 5.3]{hmp_finite}, it is enough to show that the cardinalities of $\Delta$ and $\insstarstab(R)$ are equal.

We proceed by induction on $n:=|\Max(R)|$; if $n=1$ the claim follows from Lemma \ref{lemma:intersez-valuation}. Suppose it holds up to $n-1$.

Let $\Theta$ be the standard decomposition of $R$. If $\Theta$ is not trivial, then by the inductive hypothesis the claim holds for every member of $\Theta$; by Theorem \ref{teor:star-jaffard}, $M\in\Max(R)$ is divisorial over $R$ if and only if $MT$ is divisorial over $T$ (where $T\in\Theta$ is such that $MT\neq T$), and thus $|\mathcal{M}_R|=\sum_{T\in\Theta}|\mathcal{M}_T|$. Since, by Theorem \ref{teor:jaffard-corresp}, we have $\insstarstab(R)\simeq\prod\{\insstarstab(T)\mid T\in\Theta\}$, it follows that the claim holds also for $R$.

Suppose $\Theta$ is trivial: then, $\Jac(R)$ must contain a nonzero prime ideal $P$ (and, by Lemma \ref{lemma:prufer-jac}, we can suppose $P$ is maximal with these properties). By Proposition \ref{prop:star-semistar_stab}, $|\insstarstab(R)|=|\insstarstab(R/P)|$; moreover, by Proposition \ref{prop:star-semistar}  $\mathcal{M}_R$ and $\mathcal{M}_{R/P}$ have the same cardinality. By the maximality of $P$, $R/P$ has a nontrivial standard decomposition; by induction, the claim holds for every member of the decomposition, and thus, with the same reasoning as above, we see that $|\insstarstab(R/P)|=2^{|\mathcal{M}_{R/P}|}$. Putting all together we have $|\insstarstab(R)|=2^{|\mathcal{M}_R|}$ and so $\insstarstab(R)=\Delta$ holds for every semilocal Pr\"ufer domain.

If $R$ is locally finite and finite-dimensional, then $\insstarstab(R)=\prod\{\insstarstab(T)\mid T\in\Theta\}$, where $\Theta$ is the standard decomposition of $R$. Each $T\in\Theta$ is semilocal, and thus we can apply the previous part of the proof; the claim follows.
\end{proof}

\begin{prop}\label{prop:intersez-prufer}
Let $R$ be a Pr\"ufer domain and suppose that $R$ is either:
\begin{enumerate}[(a)]
\item semilocal; or
\item locally finite and finite-dimensional.
\end{enumerate}
Then, the following are equivalent:
\begin{enumerate}[(i)]
\item $R$ is $h$-local;
\item every star operation on $R$ distributes over arbitrary intersections;
\item every star operation on $R$ distributes over finite intersections;
\item the $v$-operation on $R$ distributes over arbitrary intersections;
\item the $v$-operation on $R$ distributes over finite intersections;
\item for every fractional ideal $I$ of $R$, $I^v=\bigcap\{(IR_M)^{v^{(R_M)}}\mid M\in\Max(R)\}$.
\end{enumerate}
\end{prop}
\begin{proof}
(i $\Longrightarrow$ ii) follows from Theorem \ref{teor:star-jaffard}, Lemma \ref{lemma:intersez-ritorno} and Lemma \ref{lemma:intersez-valuation}, since $\{R_M\mid M\in\Max(R)\}$ is a Jaffard family if $R$ is $h$-local. (ii $\Longrightarrow$ iii $\Longrightarrow$ v) and (ii $\Longrightarrow$ iv $\Longrightarrow$ v) are clear, while (v $\iff$ vi) follows from Proposition \ref{prop:prufer-stab}; we only have to show that (v $\Longrightarrow$ i).

Suppose (v) holds and let $\Theta$ be the standard decomposition of $R$. If $R$ is not $h$-local, then a branch $T\in\Theta$ is not local; the hypotheses on $R$ guarantee that there is a nonzero prime ideal of $T$ contained in every maximal ideal. Therefore, we can apply Lemma \ref{lemma:prufer-jac} and find a prime ideal $Q$ such that $\Jac(T/Q)$ contains no prime ideals. By Proposition \ref{prop:star-semistar}, there is an order-preserving bijection between $\insstar(T)$ and $\inssmstar(T/Q)$, where the $v$-operation on $T$ corresponds to the semistar operation $\star$ which is the trivial extension of the $v$-operation on $T/Q$.

Since $\Jac(T/Q)$ does not contain nonzero primes, $T/Q$ admits a nontrivial Jaffard family $\Lambda$; let $Z\in\Lambda$, and define $Z':=\bigcap_{W\in\Lambda\setminus\{Z\}}W=\ortog{\Lambda}(Z)$. Then, $Z$ and $Z'$ are not fractional ideals of $T/Q$, and thus $Z^\star=Z'^\star=F$, where $F$ is the quotient field of $T/Q$; on the other hand, $Z\cap Z'=T/Q$ and thus $(Z\cap Z')^\star=T/Q$.

If $\pi:T_Q\longrightarrow T_Q/QT_Q$ is the canonical quotient, it follows that $\pi^{-1}(Z)^v=\pi^{-1}(Z')^v=T_Q$, while $\pi^{-1}(Z\cap Z')^v=\pi^{-1}(T/Q)^v=T^v=T$. Since $T$ is not local, $T\neq T_Q$, and thus $v$ does not distribute over finite intersections, against the hypothesis.
\end{proof}

\section{The class group}\label{sect:jaff-InvCl}
Let $\star$ be a star operation on $R$. An ideal $I$ of $R$ is \emph{$\star$-invertible} if $(I(R:I))^\star=R$; the set of $\star$-invertible $\star$-ideals, indicated with $\Inv^\star(R)$, is a group under the natural ``$\star$-product'' $I\times_\star J\mapsto(IJ)^\star$ \cite{jaffard_systeme,griffin_vmultiplication_1967,zafrullah_tinvt,halterkoch_libro}. Any $\star$-invertible $\star$-ideal is divisorial \cite[Theorem 1.1 and Observation C]{zafrullah_tinvt} and, if $\star_1\leq\star_2$, there is a natural inclusion $\Inv^{\star_1}(R)\subseteq\Inv^{\star_2}(R)$.
\begin{prop}\label{prop:jaffard-invt}
Let $R$ be an integral domain and $\Theta$ be a Jaffard family on $R$. The map
\begin{equation*}
\begin{aligned}
\Gamma\colon \Inv^\star(R) & \longrightarrow \bigoplus_{T\in\Theta}\Inv^{\star_T}(T)\\
I & \longmapsto (IT)_{T\in\Theta}
\end{aligned}
\end{equation*}
is well-defined and a group isomorphism.
\end{prop}
\begin{proof}
Define a map
\begin{equation*}
\begin{aligned}
\widehat{\Gamma}\colon \mathcal{F}(R) & \longrightarrow \prod_{T\in\Theta}\mathcal{F}(T)\\
I & \longmapsto (IT)_{T\in\Theta}
\end{aligned}
\end{equation*}
For every $\star$-ideal $I$, $\widehat{\Gamma}(I)=(IT)$ is a sequence such that $IT$ is $\star_T$-closed. Moreover, if $I$ is $\star$-invertible, then $(I(R:I))^\star=R$ and thus $(I(R:I)T)^{\star_T}=T$, so that $IT$ is $\star_T$-invertible. Thus $\widehat{\Gamma}(\Inv^\star(R))\subseteq\prod_{T\in\Theta}\Inv^{\star_T}(T)$, and indeed $\widehat{\Gamma}(\Inv^\star(R))\subseteq\bigoplus_{T\in\Theta}\Inv^{\star_T}(T)$ since $\Theta$ is locally finite, by Theorem \ref{teor:star-jaffard}. Hence, $\Gamma$ is well-defined, since it is the restriction of $\widehat{\Gamma}$ to $\Inv^\star(R)$.

It is straightforward to verify that $\Gamma$ is a group homomorphism, and since $I=\bigcap_{T\in\Theta}IT$, we have that $\Gamma$ (or even $\widehat{\Gamma}$) is injective.

We need only to show that $\Gamma$ is surjective. Let $(I_T)\in\bigoplus_{T\in\Theta}\Inv^{\star_T}(T)$, and define $I:=\bigcap I_T$. Since $I_T=T$ for all but a finite number of elements of $\Theta$, say $T_1,\ldots,T_n$, there are $d_1,\ldots,d_n\in R$ such that $d_iI_{T_i}\subseteq T_i$. Defining $d:=d_1\cdots d_n$, we have $dI_T\subseteq T$ for every $T$, and thus $dI\subseteq\bigcap_{T\in\Theta}T=R$, so that $I$ is indeed a fractional ideal of $R$. Moreover, since $I_T$ is $\star_T$-closed, $I_T\cap R$ is $\star$-closed, and thus $I$, being the intersection of a family of $\star$-closed ideals, is $\star$-closed. It is also $\star$-invertible, since
\begin{equation*}
(I(R:I))^\star=\bigcap_{T\in\Theta}(I(R:I)T)^{\star_T}=\bigcap_{T\in\Theta}(IT(T:IT))^{\star_T}=\bigcap_{T\in\Theta}T=R.
\end{equation*}
Therefore, $(I_T)=\Gamma(I)\in\Gamma(\Inv^\star(R))$, and thus $\Gamma$ is an isomorphism.
\end{proof}

The set of nonzero principal fractional ideals forms a subgroup of $\Inv^\star(R)$, denoted by $\Prin(I)$. The quotient between $\Inv^\star(R)$ and $\Prin(R)$ is called the \emph{$\star$-class group} of $R$ \cite{anderson_generalCG_1988}, and it is denoted by $\Cl^\star(R)$. If $\star_1\leq\star_2$, there is an injective homomorphism $\Cl^{\star_1}(R)\subseteq\Cl^{\star_2}(R)$. Of particular interest are the class group of the identity star operation (usually called the \emph{Picard group} of $R$, denoted by $\Pic(R)$) and the $t$-class group, which is linked to the factorization properties of the group (see for example \cite{samuel_factoriel,bouvier_zaf_1988,zafrullah_tinvt}). The quotient between $\Cl^\star(R)$ and $\Pic(R)$ is called the \emph{$\star$-local class group} of $R$, and it is indicated by $G_\star(R)$ \cite{anderson_generalCG_1988}.
\begin{teor}\label{teor:jaffard-clgroup}
Let $R$ be an integral domain and let $\Theta$ be a Jaffard family on $R$. Then, the map
\begin{equation*}
\begin{aligned}
\Lambda\colon G_\star(R) & \longrightarrow \bigoplus_{T\in\Theta}G_{\star_T}(T)\\
[I] & \longmapsto ([IT])_{T\in\Theta}
\end{aligned}
\end{equation*}
is well-defined and a group isomorphism.
\end{teor}
\begin{proof}
By Proposition \ref{prop:jaffard-invt}, there are two isomorphisms $\Gamma^\star:\Inv^\star(R)\longrightarrow \bigoplus_{T\in\Theta}\Inv^{\star_T}(T)$ and $\Gamma^d:\Inv^d(R)\longrightarrow \bigoplus_{T\in\Theta}\Inv^{d_T}(T)$.

Consider the chain of maps
\begin{equation*}
\Inv^\star(R)\xrightarrow{\Gamma^\star}\bigoplus_{T\in\Max(T)}\Inv^{\star_T}(T) \xrightarrow{\pi}\bigoplus_{T\in\Max(T)}\frac{\Inv^{\star_T}(T)}{\Inv^{d_T}(T)}
\end{equation*}
where $\pi$ is the componentwise quotient; then, the kernel of $\pi$ is exactly $\bigoplus_{T\in\Theta}\Inv^{d_T}(T)$. However, $\Gamma^\star$ and $\Gamma^d$ coincide on $\Inv^d(R)\subseteq\Inv^\star(R)$; hence, 
\begin{equation*}
\ker(\pi\circ\Gamma^\star)=(\Gamma^d)^{-1}(\ker\pi)=\Inv^d(R).
\end{equation*}
Therefore, there is an isomorphism $\displaystyle{\frac{\Inv^\star(R)}{\Inv^d(R)}\simeq\bigoplus_{T\in\Max(T)}\frac{\Inv^{\star_T}(T)}{\Inv^{d_T}(T)}}$. However, for an arbitrary domain $A$ and an arbitrary $\sharp\in\insstar(A)$, we have $\Prin(A)\subseteq\Inv^d(A)\subseteq\Inv^\sharp(A)$, and thus
\begin{equation*}
\frac{\Inv^\sharp(A)}{\Inv^d(A)}\simeq\frac{\Inv^\sharp(A)/\Prin(A)}{\Inv^d(A)/\Prin(A)}\simeq\frac{\Cl^\sharp(A)}{\Pic(A)}=G_\sharp(A)
\end{equation*}
so that $\Lambda$ becomes an isomorphism between $G_\star(R)$ and $\bigoplus_{T\in\Theta}G_{\star_T}(T)$, as claimed.
\end{proof}

\subsection{The class group of a Pr\"ufer domain}\label{sect:prufer-Cl}
If $\star$ is a (semi)star operation, we can define the $\star$-class group by mirroring the definition of the case of star operations: we say that $I$ is $\star$-invertible if $(I(R:I))^\star=R$, and we define $\Cl^\star(R)$ as the quotient between the group of the $\star$-invertible $\star$-ideals (endowed with the $\star$-product) and the subgroup of principal ideals. Since $(R:I)=(0)$ if $I\in\inssubmod(R)\setminus\insfracid(R)$, every $\star$-invertible ideal is a fractional ideal, and thus $\Cl^\star(R)$ coincides with $\Cl^{\star'}(R)$, where $\star':=\star|_{\insfracid(R)}$ is the restriction of $\star$.

The first result of this section is that Proposition \ref{prop:star-semistar} can be extended to the class group.
\begin{prop}\label{prop:star-semistar-clgroup}
Let $R$ be a Pr\"ufer domain and let $P$ be a nonzero prime ideal of $R$ contained in every maximal ideal. Suppose also that $P\notin\Max(R)$. Let $\star\in\insstar(R)$ and let $\sharp$ be the corresponding (semi)star operation on $D:=R/P$. Then, $\Cl^\star(R)$ is naturally isomorphic to $\Cl^\sharp(D)$.
\end{prop}
\begin{proof}
Let $\pi:R_P\longrightarrow F=Q(D)$ be the quotient map, and let $I$ be a fractional ideal of $R$ contained between $P$ and $R_P$. We claim that $\pi((R:I))=(D:\pi(I))$. In fact, if $y\in\pi((R:I))$ then $y=\pi(x)$ for some $x\in(R:I)$, and thus $y\pi(I)=\pi(x)\pi(I)=\pi(xI)\subseteq\pi(R)=D$, and thus $x\in(D:\pi(I))$. Conversely, if $y\in(D:\pi(I))$ and $y=\pi(x)$ then $y\pi(I)\subseteq D$, i.e., $\pi(xI)\subseteq D$. By the correspondence between $R$-submodules of $R_P$ and $D$-submodules of $F$ we have $xI\subseteq R$ and $y\in\pi((R:I))$.

Let $J=\pi(I)$ be a $\sharp$-invertible ideal of $D$. Then, $(J(D:J))^\sharp=D$, and thus
\begin{equation*}
\begin{array}{rcl}
R & = & \pi^{-1}\left((J(D:J))^\sharp\right)=\pi^{-1}(J(D:J))^\star=\\
& = & \left(\pi^{-1}(J)\pi^{-1}(D:J)\right)^\star=(I(R:I))^\star.
\end{array}
\end{equation*}
Therefore, $I$ is $\star$-invertible, and there is an injective map $\theta:\Inv^\sharp(D)\longrightarrow \Inv^\star(R)$. It is also straightforward to see that $\theta$ is a group homomorphism.

The well-definedness of the map $\star\mapsto\star_\phi$ implies that, if $J,J'$ are $D$-submodules of $F$, and $I:=\pi^{-1}(J)$, $I':=\pi^{-1}(J')$, then $J=zJ'$ for some $z\in F$ if and only if $I=wI'$ for some $w\in K$. Therefore, $\theta$ induces an injective map $\overline{\theta}:\Cl^\sharp(D)\longrightarrow \Cl^\star(R)$, that is clearly is a group homomorphism.

Let now $I$ be a $\star$-invertible ideal of $R$. Then, $I$ is $v$-invertible, and thus $(I:I)=R$ \cite[Proposition 34.2(2)]{gilmer}. In particular, $I$ is not a $R_P$-module, and thus the set $\valut(I)$ has an infimum $\alpha$, where $\valut$ is the valuation associated to $R_P$. If $a$ is an element of valuation $\alpha$, then $P\subsetneq a^{-1}I\subsetneq R_P$; hence, $a^{-1}I=\phi^{-1}(\phi(a^{-1}I))$ and $[I]=\overline{\theta}([\pi(a^{-1}I)])$, and in particular $[I]$ is in the image of $\overline{\theta}$. Since $I$ was arbitrary, $\overline{\theta}$ is surjective and $\Cl^\sharp(D)\simeq\Cl^\star(R)$.
\end{proof}

\begin{teor}\label{teor:clgroup-prufer}
Let $R$ be a Pr\"ufer domain,  and suppose that $R$ is either:
\begin{enumerate}[(a)]
\item semilocal; or
\item locally finite and finite-dimensional.
\end{enumerate}
Consider a star operation $\star$ on $R$. Then,
\begin{equation*}
G_\star(R)\simeq\bigoplus_{\substack{M\in\Max(R)\\ M\neq M^\star}}\Cl^v(R_M).
\end{equation*}
\end{teor}
\begin{proof}
We start by considering the case of $R$ semilocal, and we proceed by induction on the number $n$ of maximal ideals of $R$. Note that, in this case, $\Pic(R)=(0)$ and so $G_\star(R)=\Cl^\star(R)$. If $n=1$, the conclusion is trivial, since $\star\neq v$ if and only if $M\neq M^\star$.

Suppose $n>1$ and let $\Theta$ be the standard decomposition of $R$ (which is a Jaffard family by Proposition \ref{prop:Tlambda}). By Theorem \ref{teor:jaffard-clgroup}, and using the fact that $\Pic(R)=(0)=\Pic(T)$ for every $T\in\Theta$, we have $\Cl^\star(R)\simeq\bigoplus_{T\in\Theta}\Cl^{\star_T}(T)$. Moreover, since a maximal ideal $M$ of $R$ is $\star$-closed if and only if $MT$ is $\star_T$-closed, by induction it suffices to prove the theorem when the standard decomposition of $R$ is $\{R\}$.

In this case, $\Jac(R)$ contains nonzero primes, and by Lemma \ref{lemma:prufer-jac} we can find a prime ideal $Q\subseteq\Jac(R)$ such that $\Jac(R/Q)$ does not contain nonzero prime ideals. Let $A:=R/Q$.

The standard decomposition $\Theta'$ of $A$ is nontrivial, and thus every $B\in\Theta'$ is a semilocal Pr\"ufer domain with less than $n$ maximal ideals. Moreover, by Proposition \ref{prop:star-semistar-clgroup}, $\Cl^\star(R)\simeq\Cl^\sharp(A)$, where $\sharp$ is the restriction to $\insfracid(A)$ of the (semi)star operation corresponding to $\star$. Therefore, by the inductive hypothesis, 
\begin{equation*}
\Cl^\sharp(A)\simeq\bigoplus_{B\in\Theta'}\Cl^v(B)\simeq \bigoplus_{B\in\Theta'}\bigoplus_{\substack{N\in\Max(B)\\ N\neq N^{\sharp_B}}}\Cl^v(B_N)\simeq\bigoplus_{\substack{N\in\Max(A)\\ N\neq N^\sharp}}\Cl^v(A_N).
\end{equation*}
Thus, 
\begin{equation*}
\Cl^\star(R)\simeq\Cl^\sharp(A)\simeq\bigoplus_{\substack{N\in\Max(A)\\ N\neq N^\sharp}}\Cl^v(A_N).
\end{equation*}
However, if $M$ is the maximal ideal of $R$ which corresponds to the maximal ideal $N$ of $A$, then $R_M/QR_M\simeq A_N$, and thus by \cite[Theorem 3.5]{afz_vclass} we have $\Cl^v(R_M)\simeq\Cl^v(A_N)$; the claim follows.

Suppose now $R$ is finite-dimensional and of finite character, and let $\Theta$ be the standard decomposition of $R$. By Lemma \ref{lemma:dimfin}, there is a bijective correspondence between $\Theta$ and the height 1 prime ideals of $R$, and every $T\in\Theta$ is semilocal. Hence, by Proposition \ref{prop:Tlambda} and by the previous case,
\begin{equation*}
G_\star(R)\simeq\bigoplus_{T\in\Theta}G_{\star_T}(T)\simeq \bigoplus_{T\in\Theta}\Cl^{\star_T}(T)\simeq \bigoplus_{T\in\Theta}\bigoplus_{\substack{M\in\Max(T)\\ M\neq M^{\star_T}}}\Cl^v(T_M).
\end{equation*}
The conclusion now follows since $T_M=R_N$ (where $N:=M\cap R$) and $N=N^\star$ if and only if $M=M^{\star_T}$.
\end{proof}

\begin{cor}\label{cor:clgroup-bezout}
Let $R$ be a Bézout domain, and suppose that $R$ is either:
\begin{enumerate}[(a)]
\item semilocal; or
\item finite-dimensional and of finite character.
\end{enumerate}
Let $\star$ be a star operation on $R$. Then, 
\begin{equation*}
\Cl^\star(R)\simeq\bigoplus_{\substack{M\in\Max(R)\\ M\neq M^\star}}\Cl^v(R_M).
\end{equation*}
\end{cor}
\begin{proof}
It is enough to note that $\Pic(R)=0$ if $R$ is a Bézout domain, so that $G_\star(R)=\Cl^\star(R)$ for every $\star\in\insstar(R)$, and then apply Theorem \ref{teor:clgroup-prufer}.
\end{proof}

\begin{cor}
Let $R$ be a Bézout domain, and suppose that $R$ is either
\begin{enumerate}[(a)]
\item semilocal; or
\item finite-dimensional and of finite character.
\end{enumerate}
Let $S$ be a multiplicatively closed subset of $R$. Then, there is a natural surjective group homomorphism $\Cl^v(R)\longrightarrow \Cl^v(S^{-1}R)$, $[I]\mapsto[S^{-1}I]$.
\end{cor}
\begin{proof}
Let $\Delta:=\{M\in\Max(R):M\cap S=\emptyset\}$. Then, for every $M\in\Delta$, $R_M=(S^{-1}R)_{S^{-1}M}$, and thus the isomorphism of Theorem \ref{teor:clgroup-prufer} reduces to a surjective map $\Cl^v(R)\longrightarrow \bigoplus_{M\in\Delta}\Cl^v(R_M)\simeq\Cl^v(S^{-1}R)$, where the last equality comes from the fact that the maximal ideals of $S^{-1}R$ are the extensions of the ideals belonging to $\Delta$.
\end{proof}

Therefore, under each case of Theorem \ref{teor:clgroup-prufer}, the determination of $G_\star(R)$ is reduced to the calculation of $\Cl^v(V)$, where $V$ is a valuation domain. In the case where the maximal ideal $M$ of $V$ is \emph{branched} (that is, if there is a $M$-primary ideal of $V$ different from $R$, or equivalently if there is a prime ideal $P\subsetneq M$ such that there are no prime ideal properly contained between $P$ and $M$ \cite[Theorem 17.3]{gilmer}), this group has been calculated in \cite[Corollaries 3.6 and 3.7]{afz_vclass}. Indeed, if $P$ is the prime ideal directly below $M$, and $H$ is the value group of $V/P$ (represented as a subgroup of $\insR$), then
\begin{equation*}
\Cl^v(V)\simeq\begin{cases}
0 & \text{~if~}G\simeq\insZ\\
\insR/H & \text{~otherwise}.
\end{cases}
\end{equation*}

In particular, we have the following.
\begin{cor}
Let $R$ be a Bézout domain, and suppose that $R$ is either:
\begin{enumerate}[(a)]
\item semilocal; or
\item finite-dimensional and of finite character.
\end{enumerate}
For every $\star\in\insstar(R)$, $\Cl^\star(R)$ is an injective group (equivalently, an injective $\insZ$-module).
\end{cor}
\begin{proof}
By Corollary \ref{cor:clgroup-bezout} and the previous discussion, $\Cl^\star(R)\simeq\bigoplus\insR/H_\alpha$, for a family $\{H_\alpha:\alpha\in A\}$ of additive subgroups of $\insR$. Each $\insR/H_\alpha$ is a divisible group, and thus so is their direct sum; however, a divisible group is injective, and thus so is $\Cl^\star(R)$.
\end{proof}

We end with a result similar in spirit to Proposition \ref{prop:hloc-pruf-somma}.
\begin{prop}\label{prop:pruf-somma-invt}
Let $R$ be a Pr\"ufer domain and suppose that $R$ is either:
\begin{enumerate}[(a)]
\item semilocal; or
\item finite-dimensional and of finite character.
\end{enumerate}
Let $\star\in\insstar(R)$. If $I,J\in\Inv^\star(R)$, then $I+J\in\Inv^\star(R)$.
\end{prop}
\begin{proof}
Suppose first that $R$ is semilocal, and proceed by induction on $n:=|\Max(R)|$. If $n=1$, then $R$ is a valuation domain and $I+J$ is equal either to $I$ or to $J$, and the claim is proved.

Suppose the claim is true up to rings with $n-1$ maximal ideals, let $|\Max(R)|=n$ and consider the standard decomposition $\Theta$ of $R$. By Proposition \ref{prop:jaffard-invt}, $I+J\in\Inv^\star(R)$ if and only if $(I+J)T\in\Inv^{\star_T}(T)$ for every $T\in\Theta$; therefore, if $\Theta$ is not trivial, we can use the inductive hypothesis. Suppose $\Theta$ is trivial: then, $\Jac(R)$ contains nonzero prime ideals, and by Lemma \ref{lemma:prufer-jac} there is a nonzero prime ideal $Q\subseteq\Jac(R)$ such that $\Jac(R/Q)$ does not contain nonzero primes. By Proposition \ref{prop:star-semistar-clgroup}, $I/Q$ and $J/Q$ are $\sharp$-invertible $\sharp$-ideals of $R/Q$ (where $\sharp$ is the (semi)star operation induced by $\star$), and in particular $I/Q$ and $J/Q$ are fractional ideals of $R/Q$.

By construction, $R/Q$ admits a nontrivial Jaffard family $\Lambda$: for every $U\in\Lambda$, $(I/Q)U$ and $(J/Q)U$ are $\sharp_U$-invertible $\sharp_U$-ideals, and thus by the inductive hypothesis so is $(I/Q)U+(J/Q)U=((I+J)/Q)U$. Hence $(I+J)/Q$ is a $\sharp$-invertible $\sharp$-ideal, and so $I+J$ is a $\star$-invertible $\star$-ideal, i.e., $I+J\in\Inv^\star(R)$.

If now $R$ is locally finite and finite-dimensional, we see that if $\Theta$ is the standard decomposition of $R$ then every $T\in\Theta$ is semilocal. The ideal $I+J$ is $\star$-invertible if and only if $(I+J)T$ is $\star_T$-invertible for every $T\in\Theta$; however, since $IT$ and $JT$ are $\star_T$-invertible $\star_T$-ideals, the previous part of the proof shows that so is $IT+JT=(I+J)T$. Therefore, $I+J\in\Inv^\star(R)$.
\end{proof}

\section{Acknowledgements}
The author would like to thank the referee for his/her careful reading of the manuscript and his/her suggestions.

%
%

\end{document}